\documentclass{amsart}
\usepackage{amssymb}
\usepackage{graphicx}
\usepackage{amscd}
\usepackage{tikz}
\allowdisplaybreaks

\newcommand{\ignore}[1]{}

\numberwithin{figure}{section}
\numberwithin{table}{section}

\newcommand\aint{{\int \hspace{- 10pt}- \hspace{- 5 pt}}}
\newcommand\tr{\operatorname{tr}}

\newcommand\diam{\operatorname{diam}}

\newcommand\vol{\mathsf{vol}}

\newcommand\R{\mathbb{R}}

\newcommand\B{{\mathcal B}}
\newcommand\C{{\mathcal C}}

\newcommand\I{{\mathcal I}}

\renewcommand\P{{\mathcal P}}
\newcommand\Q{\mathcal Q}

\newcommand\M{{\mathcal M}}
\newcommand\T{{\mathcal T}}
\renewcommand\S{{\mathcal S}}

\newcommand{\0}{\mathaccent23}

\newcommand\Ball{\mathfrak{B}}

\newcommand\Alt{\operatorname{Alt}}
\DeclareMathOperator{\sign}{sign}

\numberwithin{equation}{section}

\newtheorem{thm}{Theorem}[section]
\newtheorem{prop}[thm]{Proposition}
\newtheorem{lem}[thm]{Lemma}

\newenvironment{remark}[1]
{\medskip \noindent {\bf Remark.} #1}

\begin{document}

\title[The bubble transform]{The bubble transform and the de Rham complex}
\thanks{The research leading to these results has received funding from the  European Research Council under the European Union's Seventh Framework Programme (FP7/2007-2013) / ERC grant agreement 339643. }
\author{Richard S. Falk}
\address{Department of Mathematics,
Rutgers University, Piscataway, NJ 08854}
\email{falk@math.rutgers.edu}
\thanks{}
\author{Ragnar Winther}
\address{Department of Mathematics,
University of Oslo, 0316 Oslo, Norway}
\email{rwinther@math.uio.no}
\date{February 4, 2022}
\thanks{}

\begin{abstract}
  The purpose of this paper is to discuss a generalization of the
  bubble transform to differential forms. The bubble transform was
  discussed in \cite{bubble-I} for scalar valued functions, or
  zero-forms, and represents a new tool for the understanding of
  finite element spaces of arbitrary polynomial degree. The present
  paper contains a similar study for differential forms.  From a
  simplicial mesh $\T$ of the domain $\Omega$, we build a map which
  decomposes piecewise smooth $k$ forms into a sum of local bubbles
  supported on appropriate macroelements.  The key properties of the
  decomposition are that it commutes with the exterior derivative and
  preserves the piecewise polynomial structure of the standard finite
  element spaces of $k$-forms.  Furthermore, the transform is bounded
  in $L^2$ and also on the appropriate subspace consisting of
  $k$-forms with exterior derivatives in $L^2$.
  \end{abstract}

\maketitle

\section{Introduction}
\label{sec:intro}
The bubble transform for scalar functions, or zero forms, was
presented in \cite{bubble-I}. In this paper, we will generalize this
construction to differential forms. More precisely, our goal is to
extend the construction of the bubble transform to the complete de
Rham complex.  Potentially, our results will have a number of
applications for the analyses of finite element methods of high
polynomial degree, such as for domain decomposition methods and the
construction of uniformly bounded projection operators.  In fact, our
techniques can also be adopted to the setting of mesh refinements, and
as a consequence, it may also be possible to obtain
results for general $hp$--methods. However, to make the present paper
as simple as possible, we will, throughout this paper, restrict the
discussion to the basic properties of the transform, without
considering possible applications.

Throughout this paper, $\Omega$ will be a bounded polyhedral domain in
$\R^n$, and for $0 \le k \le n$, we will use $\Lambda^k(\Omega)$ to
denote the space of smooth differential $k$--forms on $\Omega$. If
$\T$ is a simplicial triangulation of $\Omega$, we will use
$\Lambda^k(\T)$ to denote the space of $k$--forms on $\Omega$ which
are piecewise smooth with respect to $\T$. More precisely, the
elements of $\Lambda^k(\T)$ are smooth on the closed simplices $T$ in
the triangulation and have single-valued traces on each subsimplex of $\T$.  We
denote by $\Delta(\T)$ the set of all subsimplices of $\T$, while
$\Delta_m(\T)$ is the set of simplices of dimension $m$.  For each $f
\in \Delta(\T)$, the macroelement $\Omega_f$ consists of the union of
all $n$--simplexes in $\Delta(\T)$ containing $f$ as a subsimplex.
Furthermore, $\T_f$ is the restriction of the mesh $\T$ to the
macroelement $\Omega_f$, and $\0 \Lambda^k(\T_f)$ is the subspace of
$\Lambda^k(\T_f)$ consisting of forms with vanishing trace on the part
of the boundary of $\Omega_f$ that is in the interior of $\Omega$.

In the setting of finite element exterior calculus, there are two
fundamental families of piecewise polynomial subspaces of
$\Lambda^k(\T)$. These are the spaces $\P_r\Lambda^k(\T)$ and
$\P_r^-\Lambda^k(\T)$, where $r \ge 1$.  The spaces
$\P_r\Lambda^k(\T)$ consist of all piecewise polynomial $k$-forms
of degree $r$, while the spaces $\P_r^-\Lambda^k(\T)$ consist of
piecewise polynomial $k$-forms which locally on each subsimplex contain
$\P_{r-1}\Lambda^k$, but are contained in $\P_r\Lambda^k$.  In the
special case $r=1$, the space $\P_1^-\Lambda^k(\T)$ is exactly the
Whitney forms associated to the mesh $\T$.  For both these families of
finite element spaces, there exist sets of degrees of freedom
associated to elements of $\Delta(\T)$ which uniquely determine the
elements of the space. More precisely, an element $u$ is uniquely
determined by functionals of the form
\begin{equation}\label{DOFs}
u \mapsto \int_f \tr_f u \wedge \eta, \quad \eta \in \P'(f,k,r), \quad f \in
\Delta(\T), \dim f \ge k,
\end{equation}
where the test space $\P'(f,k,r) \subset \Lambda^{\dim f - k}(f)$. We
refer to \cite[Chapter 7]{FEEC-book}, \cite[Chapter 4]{acta},
\cite[Theorem 5.5]{bulletin}, or \cite{decomp} for
more details. The degrees of freedom of the form \eqref{DOFs}
correspond to a decomposition of the dual space into local subspaces,
and lead to a local basis, referred to as the dual basis for the
spaces $\P_r\Lambda^k(\T)$ and $\P_r^-\Lambda^k(\T)$.  A further
consequence is that the spaces themselves admit a decomposition of
the form
\begin{equation}\label{decomp-discrete}
V^k(\T)  = \bigoplus_{\stackrel{f \in   \Delta_m(\T)}{m \ge k}} V_f^k,
\quad  V_f^k \subset \0\Lambda^k(\T_f),
\end{equation}
where $V^k(\T)$ is a space of the form $\P_r\Lambda^k(\T)$ or
$\P_r^-\Lambda^k(\T)$, and $V_f^k$ is a
corresponding local space associated to the simplex $f$.  The space
$V_f^k$ consists of functions in $V^k(\T)$ with all degrees of freedom
taken to be zero except the ones associated to the simplex $f$.  More
precisely, a function $u \in V^k(\T)$ admits a decomposition
\[
u = \sum_{\substack{f \in   \Delta_m(\T)\\ m \ge k}} u_f, \quad u_f \in V_f^k,
\]
and the map $u \mapsto \{u_f \}$ is implicitly given by the degrees of
freedom \eqref{DOFs}.  In particular,
\begin{equation}\label{trace-prop}
\tr_f \sum_{j = k}^m u_j = \tr_f u, \quad f \in \Delta_m(\T), \,  k \le m \le n,
\end{equation}
where $ u_j = \sum_{g \in \Delta_j(\T)} u_g$ and where $\tr $ denotes
the trace operator.  The map $u \mapsto \{u_f \}$ depends heavily on
the particular space $V^k(\T)$, and in particular on the polynomial
degree $r$.  On the other hand, the geometry of the decomposition
\eqref{decomp-discrete}, represented by the macroelements $\Omega_f$
and the associated mesh $\T_f$, is independent of the choice of
discrete spaces.  This indicates that it might be possible to define
the map $u \mapsto \{u_f \}$ independent of the choice of discrete
spaces. With some modifications explained below, this is what we
achieve by the construction given in this paper.


The bubble transform $\B^k$ that we will construct is made up of local
operators $B_{m,f}^k: \Lambda^k(\T) \to \0\Lambda_m^k(\T,f)$, where $f
\in \Delta_j(\T)$, $m \le j \le m+k$, and $\0\Lambda_m^k(\T,f)$ is a 
space of rational $k$--forms with support in $\Omega_f$.
The functions in this space are piecewise
smooth, but are allowed to be singular at the boundary of $f$. The
precise definition of $\0\Lambda_m^k(\T,f)$ will be given in
Section~\ref{sec:Bmkprop} below. The corresponding sum
\begin{equation*}
B_m^k  = \sum_{\substack{f \in \Delta_{m+j}(\T)\\0 \le j \le k}} B_{m,f}^k
\end{equation*}
will be a global operator which maps the space of piecewise smooth $k$
forms, $\Lambda^k(\T)$, to itself.  In other words, the singular
components that may be present in the local functions $B_{m,f}^ku$
will cancel when we sum over all $f$. The maps $B_m^k$ will have a
trace property similar to \eqref{trace-prop}, i.e., for any $u \in
\Lambda^k(\T)$,
\begin{equation*}
\tr_f \sum_{j = 0}^m B_j^k u = \tr_f u, \quad f \in \Delta_m(\T),
 \,  k \le m \le n.
\end{equation*} 
The end result is that we can write
\begin{equation}\label{decomp-id}
u = \sum_{m=0}^n B_{m}^k u = \sum_{m=0}^n
\sum_{\substack{f \in \Delta_{m+j}(\T)\\0 \le j \le k}} B_{m,f}^k u.
\end{equation}
Since there are no subsimplexes of $\T$ of dimension greater than $n$,
the sum over $j$ above should be restricted to $0 \le j \le
n-m$. However, for simplicity we adopt the notation above throughout
the paper, where $\Delta_{m+j}(\T)$ is empty
for $j > n-m$.  To sum up,
each operator $B_{m,f}^k$ will map $u$ into a local bubble, and the
complete collection, $\B^k = \{ B_{m,f}^k \}$, produces a local
decomposition of $u$.  Although the operators $B_{m,f}^k$ will not
commute with the exterior derivative, the operators $B_{m}^k$ will
have this key property. More precisely, the diagram
\begin{equation}
\label{B-commute}
\begin{CD}
\Lambda^k(\T) @>d>> 
\Lambda^{k+1}(\T)\\
@VVB_m^k V  @VVB_m^{k+1} V\\
    \Lambda^k(\T) @>d>>
    \Lambda^{k+1}(\T)
\end{CD}
\end{equation}
commutes.  
The bubble transform also preserves the piecewise polynomial spaces
$\P_r\Lambda^k(\T)$ and $\P_r^-\Lambda^k(\T)$ in the sense that
\begin{equation}
\label{P-preserve}
B_m^k (V^k(\T)) \subset 
 V^k(\T), 
\end{equation}
where $V^k(\T)$ can be either $\P_r\Lambda^k(\T)$ or $\P_r^-\Lambda^k(\T)$.
Finally, we will show that the bubble transform is bounded in $L^2$ in
the sense that
\begin{equation}
\label{B-bound}
 \|B_{m}^k u\|_{L^2(\Omega)}, \
\Big(\sum_{j=0}^k \sum_{f \in \Delta_{m+j}(\T)} \|B_{m,f}^k u\|_{L^2(\Omega)}^2 \Big)^{1/2} 
 \le c \|u\|_{L^2(\Omega)}, 
\end{equation}
for $0 \le m \le n$, where the constant $c$ depends on the {\it shape
  regularity} constant of the mesh $\T$.  The operator $B_{m,f}^k$
will be defined by a recursive procedure.  The key tool for the
construction is a family of operators, $C_{m,f}^k$, which we will
refer to as cut--off operators, since functions in the range have
support in $\Omega_f$.
By using these operators,  $B_{m,f}^ku$ is defined
recursively by
\begin{equation}\label{def-Bm}
B_{m,f}^k u = C_{m,f}^k(u - \sum_{j = 0}^{m-1} B_j^k u), \quad m=0, 1, \ldots , n.
\end{equation}
Hence, the properties of the operators $B_{m,f}^k$ will  be derived
from corresponding properties of the operators $C_{m,f}^k$. 

 The present study is partly motivated by the $hp$-finite
  element method, i.e., where both piecewise polynomials of arbitrary
  high degree and arbitrary small mesh cells are allowed.  The analysis
  of finite element methods based on mesh refinements and a fixed
  polynomial degree, i.e., the $h$-method, is by now very well
  understood,  with a number of finite element spaces developed
    for approximating all the spaces comprising the de Rham complex,
    A key step in this analysis has been
    the development of bounded projections that commute with the
    exterior derivative, e.g., see \cite{acta}, \cite{arnold-guzman},
    \cite{MR2373181},     \cite{local-cochain},
    \cite{double-complexes}, and \cite{schoberlcp}.

    The corresponding analysis for the $p$-method, where the
    polynomial degree is unbounded, is so far less canonical.
    Pioneering results for the $p$-method applied to second order
    elliptic problems in two space dimensions were obtained by
    Babu\v{s}ka and Suri \cite{MR899702}, while a corresponding
    analysis in three dimensions can be found in \cite{MR1445738}. The
    study of the $p$-method for Maxwell equations was initiated in
    \cite{MR2034876}, and inspired the later work presented in
    \cite{MR2105164, MR2439500, MR2551195, MR2904580, melenk-rojik} on
    discretization of the de Rham complex in three space dimensions. A
    crucial step in the analysis presented in these papers is the use
    of projection-based interpolation operators, as proposed in
    \cite[Chapter 3]{MR2459075}, to construct projection operators
    which commute with the exterior derivative.  The results of these
    papers can be used to derive a number of convergence results for
    the $p$-method, including for eigenvalue problems
    \cite{MR2764424}. However, the approach using projection-based
    interpolation will usually not lead to projection operators that
    are bounded in appropriate Sobolev norms, and a common challenge
    is to show that desired bounds are independent of the polynomial
    degree.  An alternative approach to the construction of commuting
    projections is discussed in \cite{ern-g-s-v}.  These operators are
    $L^2$ bounded, but so far the construction is limited to the last
    part of the de Rham complex in two and three dimensions. A further
    discussion and additional references for interpolation operators
    and approximation in the $hp$-setting can also be found in this
    paper.

    Preconditioners based on domain decomposition for the operators
    arising from finite element approximation of second order elliptic
    equations are considered in \cite{schoberl08}. For the two-level
    Schwarz method, it is shown that the condition number is bounded
    uniformly in both the mesh size $h$ and polynomial degree.
    However, so far the problem of establishing a similar bound with
    respect to the polynomial degree for Schwarz methods applied to
    more general Hodge-Laplace problems seems to be open.

    The theory developed in this paper indicates an alternative path
    towards the understanding of finite element methods of high
    polynomial degree.  In fact, the theory presented here is
    developed without reference to any specific piecewise polynomial
    space. The setting is simply a given domain, with a given
    simplicial mesh, and all the operators defining the basic
    decompositions depend only on the domain $\Omega$ and the mesh
    $\T$.  In particular, the bounds we obtain only depend on these
    objects.  However, the relation to more specific piecewise
    polynomial spaces appears as a consequence of the invariance
    property expressed by \eqref{P-preserve}.  The discussion in the
    present paper is restricted to basic properties of the bubble
    transform, without considering possible applications to more
    specific problems related to finite element methods.  However, the
    use of the theory presented in this paper to analyze domain
    decomposition methods and to construct projections that
    commute with the exterior derivative appears to be a promising new
    approach, although not a straightforward one.


  This paper is organized as follows. In Section~\ref{sec:prelims} we
  introduce some basic notation and present some of the tools we will
  need for the construction. In particular, in
  Section~\ref{sec:Bmkprop}, we will show how the main results will
  follow from corresponding properties of the cut--off operators
  $C_{m,f}^k$.  As a consequence, the rest of the paper will be
  devoted almost entirely to analysis of these cut--off operators.  A
  brief review of some results for scalar valued functions, or
  zero--forms, is given in Section~\ref{sec:k=0}, while
  Section~\ref{sec:primalop} contains a preliminary discussion of
  corresponding results for $k$ forms. In particular, this discussion
  motivates the need for a new family of order reduction operators
  which will be defined and analyzed in Section~\ref{sec:Refk}. These
  operators comprise a new tool developed in this paper, and their
  construction is based on the double complex idea introduced in
  \cite{local-cochain, double-complexes}; see also
  \cite{arnold-guzman}.  Using the order reduction operators, the
  general definition of the operators $C_{m,f}^k$ will then be given
  in Section~\ref{Cmkops}.  Section~\ref{sec:poly-pre} is devoted to
  invariance properties, i.e., we derive the key results leading to
  the invariance property \eqref{P-preserve} and the commuting
  relation \eqref{B-commute}. At the end of that section, we briefly
  consider a possible approach for constructing projection operators,
  with desired properties, that are defined from local projections
  into pure polynomial spaces. Finally, in Section~\ref{sec:bounds},
  we verify the basic bounds in appropriate operator norms.

\section{Preliminaries and the main results}
\label{sec:prelims} 

\subsection{Assumptions}
\label{sec:assumptions}
Throughout the paper we assume that $\Omega$ is a bounded polyhedral
domain in $\R^n$ which is partitioned into a finite set of $n$
simplexes.  Furthermore, the simplicial triangulation $\T$, frequently
referred to as a mesh, is assumed to be a simplicial decomposition of
$\Omega$, i.e., the union of these simplices is the closure of
$\Omega$, and the intersection of any two is either empty or a common
subsimplex of each. As in \cite{local-cochain} ,
  cf. also \cite{arnold-guzman}, we will assume that the extended macroelement
    $\Omega_f^E$, defined by
\begin{equation*}
\Omega_f^E = \cup_{i \in I(f)} \Omega_{x_i},
\end{equation*}
is contractible for all $f \in \Delta(\T)$.
Finally, in the beginning of Section~\ref{sec:bounds} we will make an additional 
topological assumption on the mesh $\T$ which will be used to obtain the bound \eqref{B-bound}.


\subsection{Notation}
\label{sec:notation}
We start by recalling some standard notation for differential forms.
If $u \in \Lambda^k(\Omega)$, the space of smooth $k$ forms on the domain
$\Omega$, the {\it exterior derivative} $d = d^k  : \Lambda^k(\Omega) \to 
\Lambda^{k+1}(\Omega)$ is given by
\begin{equation*}
d u_x(v_1, \ldots, v_{k+1}) = \sum_{j=1}^{k+1} (-1)^{j+1} \partial_{v_j}
u_x(v_1, \ldots, \hat v_j, \ldots, v_{k+1}),
\end{equation*}
where the hat is used to indicate a suppressed argument and the
vectors $v_j$ are elements of the corresponding tangent space
$T(\Omega) = \R^n$.  If $u^1 \in \Lambda^j(\Omega)$ and $u^2 \in
\Lambda^k(\Omega)$, then the wedge product, $u^1 \wedge u^2$, is a
corresponding form in $\Lambda^{j+k}(\Omega)$ given by
\begin{equation*}
(u^1 \wedge u^2)(v_1, \ldots, v_{j+k})
= \sum_{\sigma} (\sign \sigma) u^1(v_{\sigma(1)}, \ldots,
v_{\sigma(j)}) u^2(v_{\sigma(j+1)}, \ldots,
v_{\sigma(j+k)}),
\end{equation*}
where the sum is over all permutations $\sigma$ of $\{1, \ldots,
j+k\}$, for which $\sigma(1) < \sigma(2) < \cdots < \sigma(j)$ and
$\sigma(j+1) < \sigma(j+2) < \cdots < \sigma(j+k)$.  We will use
$\lrcorner$ to denote contraction, i.e., if $u \in \Lambda^k(\Omega)$
and $v = v(x)$ is a vector field, then $u \lrcorner v$ denotes the
$k-1$ form such that
\[
(u \lrcorner v)_x(v_1, \ldots , v_{k-1}) = u_x(v(x),v_1, \ldots , v_{k-1}).
\]
A smooth map $F: \M \to \M^\prime$ between manifolds
provides a pullback of a differential form from $\M^\prime$ to $\M$, 
i.e., a map from $\Lambda^k(\M^\prime) \to \Lambda^k(\M)$ given by
\begin{equation*}
(F^* u)_x(v_1, \ldots, v_k) 
= u_{F(x)}(D F_x(v_1), \ldots,  D F_x(v_k)).
\end{equation*}
The pullback respects exterior products and differentiation, i.e.,
\begin{equation*}
F^*(u^1 \wedge u^2) = F^* u^1 \wedge F^* u^2, \qquad
F^*(d u) = d (F^* u).
\end{equation*}
In the special case when $\M$ is a submanifold of $\M^\prime$, then
the pullback of the inclusion map, $\Lambda^k(\M^\prime) \to
\Lambda^k(\M)$, is the trace map $\tr_{\M}$.  We will use $H
\Lambda^k(\Omega)$ to denote the Sobolev space given by
\begin{equation*}
H \Lambda^k(\Omega) = \{u \in L^2 \Lambda^k(\Omega) :
d u \in L^2 \Lambda^{k+1}(\Omega)\},
\end{equation*}
where $L^2 \Lambda^k(\Omega)$ is the space of $k$-forms with values in
$L^2$.  As a consequence of the identity $d \circ d = 0$, we obtain
the de Rham domain complex given by
\begin{equation*}
0\rightarrow H \Lambda^0(\Omega)  \xrightarrow{d} 
H \Lambda^1(\Omega) \xrightarrow{d} \cdots
\xrightarrow{d} H\Lambda^n(\Omega) \to 0.
\end{equation*}
We recall from the introduction above that $\Lambda^k(\T)$ denotes the
corresponding space of piecewise smooth $k$ forms with single valued
traces. Then $\Lambda^k(\T) \subset H\Lambda^k(\Omega)$.  Furthermore,
the piecewise polynomial space $\P_r\Lambda^k(\T)$ is the set of
elements $u$ of $\Lambda^k(\T)$ such that for fixed tangent vectors
$v_1, \ldots ,v_k$, the scalar function
\[
u(v_1,\ldots ,v_k) \in \P_r(T), \quad T \in \Delta_n(\T),
\]
where $\P_r(T)$ denote the set of scalar valued polynomials of degree
less than or equal to $r$ on $T$.  Finally, the space $\P_r^-\Lambda^k(\T)$ is
the space of functions $u$ in $\P_r\Lambda^k(\T)$ such that
\[
(u\lrcorner v)(v_1,\ldots ,v_{k-1}) \in \P_r(T), \quad T \in \Delta_n(\T)
\]
for any vector field $v$ of the form $v(x) = x-a$, where $a \in \R^n$ is fixed.
To summarize the relation between the spaces just introduced, we can state 
\[
\P_r^-\Lambda^k(\T) \subset \P_r\Lambda^k(\T) \subset \Lambda^k(\T) 
\subset H\Lambda^k(\Omega).
\]
We recall that $\Delta_0(\T)$ is the set of simplices of
dimension zero, i.e., the set of vertices of $\T$.  We will assume
that the vertices are numbered by a set of integers $\I = \{ 0, 1,
\ldots ,N(\T) \}$ such that
\[
\Delta_0(\T) = \{ x_i \, : \, i \in \I \, \},
\]
and leading to an ordering of the vertices.  The barycentric
coordinate associated to $x_i \in \Delta_0(\T)$ is denoted $\lambda_i(x)$.
In other words, $\lambda_i$ is the piecewise linear function equal to
one at $x_i$ and zero at all other vertices.  Any subset $f$ of
$\Delta_0(\T)$ corresponds to a set of integers $I(f) \subset \I$.  The
number of elements in $f$ is denoted $|f|$.
In particular, $f \in \Delta_m(\T)$ is an ordered subset of
$\Delta_0(\T)$.   
We will use the notation $[ \cdot, \ldots ,\cdot ] $ to
denote convex combinations, such that if $f \in \Delta_m(\T)$ with $I(f) =
\{\, 0,1,\ldots ,m \, \}$ then $ f = [x_0, x_1, \ldots x_m]$.
Furthermore, the statement $g \in \Delta(f)$ means that $g$ is a
subcomplex of $f$ with ordering inherited from $f$. The set $\bar
\Delta(f)$ contains the emptyset, $\emptyset$, in addition to the
elements of $\Delta(f)$, and $\emptyset $ is the single element of
$\Delta_{-1}(\T)$.  If $f \in [x_{j_0}, x_{j_1}, \ldots
x_{j_m}] \in \Delta_m(\T)$ then
\[
\sigma_f(x_{j_i}) = i .
\]
In other words, $\sigma_f(y)$ gives the internal numbering of $y$ for
a vertex $y$ of the simplex $f$.  For any $f \in \bar \Delta(\T)$ the
piecewise linear function $\rho_f = \rho_f(x)$, defined by
\[
\rho_f(x) = 1 - \sum_{i \in I(f)} \lambda_i(x),
\]
can be seen as a distance function between $f$ and $x \in
\Omega$. Note that $0 \le \rho_f(x) \le 1$ and $\rho_f \equiv
1$ if $f = \emptyset$.
Recall that for each $f \in \Delta(\T)$, the corresponding macroelement
$\Omega_f$ is defined as the union of all elements of $\Delta_n(\T)$
containing $f$. Alternatively, the interior of $\Omega_f$ is the set
\[
\{ \, x \in \Omega \, : \, \lambda_i(x) > 0, \, i \in I(f) \, \}.
\]
As a consequence, if $f \in \Delta_m(\T)$ and $g \in \Delta_j(\T)$ for $j <
m$, then $g$ will not belong to the interior $\Omega_f$. Furthermore,
if $g \in \Delta(f)$, then $\Omega_f \subset \Omega_g$.  
If $f \in \Delta_m(\T)$, then $\phi_f$ will
denote the Whitney form associated to $f$.  More precisely, if $f =
[x_{j_0}, x_{j_1}, \ldots x_{j_m}]$ then $\phi_f$ is given by
\[
\phi_f = \sum_{i =0}^m (-1)^{i} \lambda_{j_i} d\lambda_{j_0} \wedge 
\ldots \wedge \widehat{d\lambda_{j_i} } \wedge
\ldots \wedge d\lambda_{j_m},
\]
and 
\begin{equation}\label{d-phi}
d\phi_f = (m+1) d\lambda_{j_0} \wedge \ldots \wedge d\lambda_{j_m}.
\end{equation}
The functions $\{ \phi_f \}_{f \in \Delta_m}$ span the space 
$\P_1^-\Lambda^m(\T)$, and they are local with support in $\Omega_f$.

We define a simplex $\S = \S(\T)$ by
\begin{equation*}
\S = \big\{\, \lambda = (\lambda_0, \dots, \lambda_{N}) \in \R^{N+1} \, : \,
\sum_{j =0}^{N} \lambda_j = 1, \quad \lambda_j \ge 0 \,  \big\},
\end{equation*}
where $N = N(\T)$.  The relevance of this simplex can be understood by
introducing the barycentric map $L$ given by $L : \Omega \to \S$ by
$L(x) = (\lambda_0(x), \lambda_1(x)\ldots \lambda_N(x))$.  In fact, if
$N >> n$ then the range of the barycentric map $L$ will only cover
parts of the boundary of the huge simplex $\S$.  However, for
notational simplicity, we have found it convenient to introduce the
simplex $\S$.  We let $\S^c$ be the set of all convex combinations
between $\S$ and the origin, i.e., $\S^c = [0,\S]$. Alternatively,
\[
S^c = \big\{\, \lambda = (\lambda_0, \dots, \lambda_{N}) \in \R^{N+1} \, : \,
\sum_{j =0}^{N} \lambda_j \le  1, \quad \lambda_j \ge 0 \,  \big\}.
\]
Furthermore, 
for any $f \in \bar \Delta$, the mapping $L_f : \Omega \to \S^c$ is defined by 
\[
(L_f(x))_i = \lambda_i(x), \, \,  i \in I(f), 
\quad (L_f(x))_i = 0, \, \,  i \in \I \setminus  I(f).
\]
Note that for $f = \emptyset$, $(L_f(x))_i = 0$ for all $i \in \I$,
while for any $f \in \Delta_m(\T)$ the range of the map $L_f$ is a
subcomplex of $S^c$ with dimension $m+1$. We will denote this
subcomplex of $\S^c$ by $\S_f^c$, and $\S_f = \S_f^c \cap \S$.  In
fact, $\S_f^c$ is the convex set with the origin and the endpoints of
the coordinate vectors $\{e_i, \, i \in I(f) \}$ as extreme points,
where $e_i$ denotes a coordinate vector in $\R^{N+1}$.  In the
construction below, we will frequently use the pullback $L_f^*$
mapping $\Lambda^k(\S_f^c)$ to $\Lambda^k(\T)$, i.e., $L_f^*$ maps
smooth forms on $\S_f^c$ to piecewise smooth forms on $\Omega$.
Similarly, it will also map polynomial forms to piecewise polynomial
forms.  For $\lambda \in \S^c$ we let
\begin{equation*}
b(\lambda) = 1 - \sum_{j =0}^{N} \lambda_j.
\end{equation*}
Hence, $b(\lambda)$ measures the distance from $\lambda \in \S^c$ to $\S$, and 
$\rho_f(x) = b(L_f(x))$.
If $f \in \Delta_m(\T)$ and $T \in \Delta_n(\T_f)$, we let $f^*(T)
\in \Delta_{n-m-1}(T)$ be the face opposite $f$.  Alternatively,
\begin{equation*}
f^*(T) = \{\, x \in T \, : \,  \lambda_j(x) =0, \, 
j \in I(f) \, \}.
\end{equation*}
We then define 
\[
f^* = \bigcup_{T \in \Delta_n(\T_f)} f^*(T).
\]
The set $f^*$ can be viewed as an $n-m-1$ dimensional manifold
composed of the simplexes $f^*(T)$, and all elements of $\Omega_f$ can
be written uniquely as a convex combination of the points $x_i, \, i
\in I(f)$ and a point $q_f(x) \in f^*$. In fact, if $x \subset T \in
\Delta_n(\T_f)$, then
\[
q_f(x) = \Big(\sum_{j\in I(f^*(T))} \lambda_j(x)x_j\Big)/
\Big(\sum_{j \in I(f^*(T))} \lambda_j(x)\Big),
\]
and 
\[
x = \sum_{i \in I(f)} \lambda_i(x) x_i + \rho_f(x) q_f(x).
\]
In the special case when $m=n-1$, the manifold $f^*$ will be reduced
to two vertices, or only one close to the boundary, while in the case
$f = \emptyset$ we have $\Omega_f = f^* = \Omega$ and $q_f(x) = x$.

\subsection{The average operators}
\label{sec:average}
A key tool for our construction below is a family of average
operators, $A_f^k$, where $f \in \Delta$, which will map elements of
$\Lambda^k(\T_f)$ to $\Lambda^k(S_f^c)$. In other words, these
operators map piecewise smooth $k$--forms on $\Omega_f$ to smooth
$k$--forms on $\S_f^c$.  The operators $A_f^k$ will be defined by a
function $G= G(y,\lambda)$ given by
\[
G(y,\lambda) = \sum_{i \in \I} \lambda_i x_i + b(\lambda)y,
\]
where $y \in \Omega$ and $\lambda\in \S^c$. Note that if $x \in f$
then, since $b(L_fx) = 0$, we have
\begin{equation}\label{G-L_f-rel}
G(y, L_f x) = x, \quad x \in f.
\end{equation}
In fact, we will only consider the function $G$ for $y \in \Omega_f$
and $\lambda \in \S_f^c$ for some simplex $f \in \Delta$. In this case,
we will have $G(\lambda,y) \in \Omega_f$, i.e., we can regard $G$ as a
map $G : \Omega_f \times \S_f^c \to \Omega_f$.  We note that for a
fixed $y$, $G(y, \lambda)$ is linear with respect to $\lambda$.  The
corresponding derivative with respect $\lambda$, $DG(y, \cdot)$, is
therefore an operator mapping tangent vectors of $\S_f^c$,
$T(\S_f^c)$, into $T(\Omega_f)$ which is independent of $\lambda$.  It
is given by
\[
DG(y, \cdot) = \sum_{i \in I(f)} (x_i - y) d\lambda_i.
\]
For each fixed $y \in \Omega_f$, the map $G(y, \cdot)$ maps $\S_f^c$
to $\Omega_f$. Therefore, the corresponding pullback, $G(y, \cdot)^*$,
maps $\Lambda^k(\Omega_f)$ to $\Lambda^k(\S_f^c)$.  As a further
consequence, the average of these maps over $\Omega_f$ with respect to
$y$ will also map $\Lambda^k(\Omega_f)$ to $\Lambda^k(\S_f^c)$.  The
operator $A_f^k$ is defined by
\[
A_f^k u  = \aint_{\Omega_f} G(y, \cdot)^*u \, dy
= \frac{1}{|\Omega_f|} \int_{\Omega_f} G(y, \cdot)^*u \, dy
\] 
or more precisely,
\[
(A_f^k u)_{\lambda}(v_1, \ldots ,v_k)  
= \aint_{\Omega_f} u_{G(y, \lambda)}(DG(y, \cdot) v_1, \ldots, DG(y, \cdot )v_k) 
\, dy,
\]
where $v_1, \ldots ,v_k \in T(\S_f^c)$.  Note that since pullbacks
commute with the exterior derivative, so do the operators $A_f^k$,
i.e., $d A_f^k u = A_f^{k+1} du$.  Other key properties of the
operators $A_f^k$, stated in the lemma below, are that it maps
piecewise smooth forms to smooth forms, it maps piecewise polynomial
forms to polynomial forms, and it is trace preserving.

\begin{lem}\label{lem:A-prop}
Let $f \in \Delta_m(\T)$. The operators $A_f^k$ satisfy
\begin{itemize}
\item[i)] $A_f^k (\Lambda^k(\T_f)) \subset \Lambda^k(\S_f^c)$,
\item[ii)]  $A_f^k (\P_r\Lambda^k(\T_f) ) \subset \P_r\Lambda^k(\S_f^c)$ and
$A_f^k (\P_r^-\Lambda^k(\T_f) ) \subset \P_r^-\Lambda^k(\S_f^c)$,
\item[iii)] $\tr_{f} L_f^*A_f^k u = \tr_f u$ for $u \in \Lambda^k(\T_f)$, $k \le m \le n$.
\end{itemize}
\end{lem}

\begin{proof} 
  Assume that $u \in \Lambda^k(\T_f)$.  From the definition of the
  operator $A_f^k$, we obtain
\[
(A_f^k u)_{\lambda}(v_1, \ldots ,v_k)  = |\Omega_f|^{-1}\sum_{T \in \Delta_n(\T_f)}
\int_{T} u_{G(y, \lambda)}(DG(y, \cdot) v_1, \ldots, DG(y, \cdot )v_k) \, dy,
\]
where $|\Omega_f|$ denote the volume of $\Omega_f$.  Also observe that
if we fix $y \in \Omega_f$, then the subset of $\Omega_f$ given by
\[
\{ \, G(y, \lambda) \, : \, \lambda \in \S_f^c \, \}
\]
belongs to a single $n$ simplex of $\Omega_f$.  Therefore, since $G(y,
\cdot)$ is a smooth function of $\lambda$ and $u$ is piecewise smooth,
we can conclude that for each fixed $y$, the integrand appearing
in the definition of $(A_f^k u)_{\lambda}$ varies smoothly with
$\lambda$. As a consequence, $(A_f^k u)_{\lambda}(v_1, \ldots ,v_k)$
is a smooth function of $\lambda$, and therefore part i) is
established.

If $ u \in \P_r\Lambda^k(\T_f)$, then the integrand 
\[
u_{G(y, \lambda)}(DG(y, \cdot) v_1, \ldots, DG(y, \cdot )v_k) \in \P_r(\S_f^c)
\]
as a function of $\lambda$. The same is true for the integral with
respect to $y$, so $A_f^k u \in \P_r\Lambda^k(\S_f^c)$.  To show the
corresponding preservation of the $\P_r^-$ spaces, we have to show that
$(A_f^k u )\lrcorner \lambda \in \P_r\Lambda^{k-1}(\S_f^c)$ for $ u
\in \P_r^-\Lambda^k(\T_f)$.  However, from the fact that
\[
DG(y, \cdot) \lambda = \sum_{i \in I(f)} \lambda_i(x_i - y) = 
G(y, \lambda) - y,
\]
we obtain
\begin{multline*}
((A_f^k u ) \lrcorner \lambda)_{\lambda} (v_1, \ldots ,v_{k-1})
\\
= \aint_{\Omega_f} (u_{G(y, \lambda)} \lrcorner (G(y,\lambda) - y))
(DG(y, \cdot) v_1, \ldots, DG(y, \cdot )v_{k-1}) \, dy.
\end{multline*}
If $u \in \P_r^-\Lambda^k(\T_f)$, we have that $u_x \lrcorner (x- y)$
is an element of $\P_r\Lambda^{k-1}(\T_f)$ as a function of $x$ for
each fixed $y$, and therefore, by the linearity of $G(\lambda,y)$ with
respect to $\lambda$, we can conclude that the integrand above is in
$\P_r(\S_f^c)$. In other words, we have established that $A_f^k u \in
\P_r^-\Lambda^k(\S_f^c)$.

Finally, we have to show the trace property. However, for each fixed $y$,
\[
L_f^* \circ G(y , \cdot)^* = (G(y , \cdot) \circ L_f)^*,
\]
and by \eqref{G-L_f-rel}, the function $G(y , \cdot) \circ L_f =
G(y,L_f \cdot)$ is the identity on $f$.  We can therefore conclude
that
\[
\tr_f L_f^* A_f^k u = \tr_f \aint_{\Omega_f} (G(y , \cdot) \circ L_f)^*u \, dy 
= \tr_f u.
\]
This completes the proof of  the lemma.
\end{proof}

\subsection{The main results}
\label{sec:Bmkprop}
We recall that the operators $B_{m,f}^k$ are related to the cut--off operators
$C_{m,f}^k$ by the iteration \eqref{def-Bm}, i.e., 
\[
B_{m,f}^k u = C_{m,f}^k(u - \sum_{j = 0}^{m-1} B_j^k u), \quad m=0, 1, \ldots , n.
\]
As a consequence, the operators $B_m^k$ will satisfy 
\begin{equation}
\label{Bmkdef}
B_{m}^k u = C_{m}^k(u - \sum_{j = 0}^{m-1} B_j^k u), \quad m=0, 1, \ldots , n,
\end{equation}
where 
\[
C_m^k = \sum_{\substack {f \in \Delta_{m+j}(\T)\\ 0 \le j \le k}} C_{m,f}^k.
\]
The purpose of this section is to show how the desired properties for
the operators $\B^k$, given by $\B^k = \{ B_{m,f}^k \}$, will follow
from corresponding properties of the cut--off operators $C_{m,f}^k$.
As a consequence, the rest of the paper will almost entirely be
devoted to analysis of the cut--off operators.

Before we state the key results for the operators $C_{m,f}^k$, we will
give a precise definition of the local space $\0\Lambda_m^k(\T ,f)$,
introduced in the introduction.  If $f \in \Delta_m(\T)$, we define
the space $\Lambda_m^k(\T ,f)$ by
\[
\Lambda_m^k(\T ,f) = \{ \, u = \sum_{\substack{g \in \Delta_{j}(f)\\0 \le j \le m-1}} 
\rho_g^{-1} w_g \,  : \, w_g \in \Lambda^k(\T) \, \}.
\]
This space consists of $k$--forms which can be expressed as a sum of
rational functions with possible singularities at the boundary of $f$.
Furthermore, we let $\0\Lambda_m^k(\T ,f)$ be the subspace of
functions which are supported on $\Omega_f$, i.e., they are
identically zero on $\Omega \setminus \Omega_f$.

The results in Lemma~\ref{lem:trCmk} below provide a summary of
results to be established in Lemmas~\ref{lem:trace-preserve-Ck-primal}
and \ref{lem:support-Ck} and part i) of
Proposition~\ref{prop:preservation}.
\begin{lem}
\label{lem:trCmk}
Let $u \in \Lambda^k(\T)$ and $f \in \Delta_{m+j}(\T)$ for $0 \le m
\le n$ and $0 \le j \le k$.  Then $C_{m,f}^k u \in
\0\Lambda_{m+j}^k(\T ,f)$, while $C_m^k u \in \Lambda^k(\T)$. 
Furthermore, if $k \le m \le n$ and $f \in \Delta_m(\T)$, then we
also have $\tr_f C_{m,f}^k u = \tr_f u$, which gives
\[
tr_f C_m^ku = \tr_f u, \quad f \in \Delta_m(\T).
\]
\end{lem}

The first part of the lemma expresses the fact that the operator
$C_{m,f}^k$ maps a piecewise smooth form into a rational differential
form with local support on $\Omega_f$, and in such a way that when we
sum over all $f \in \Delta_{m+j}(\T), \, 0 \le j \le k, $ we obtain a
form which is piecewise smooth.  The last statement, that the operator
$C_m^k$ preserves the trace of $u$ on all simplexes in $\Delta_m(\T)$,
follows from the stated trace properties of the local operators
$C_{m.f}^k$, since $f \in \Delta_m(\T)$ will not belong to the interior
of any $\Omega_{f^\prime}$ for $f^\prime \neq f$, $f^\prime \in
\Delta_{m+j}(\T)$, $0 \le j \le k$. Furthermore, for $f \in \Delta_n(\T)$,
we have $\Omega_f = f$, and by Lemma~\ref{lem:trCmk},
\[
\tr_f C_{n,f} ^k u = \tr_f u, \quad \text{and } C_{n,f}^k
\equiv 0 \quad \text{on } \Omega\setminus f.
\]
This completely specifies the operators $C_{n,f}^k$.

\begin{remark}
  If $f \in \Delta_m(\T)$ and $g \in \Delta(f)$, $g \neq f$, then $g
  \subset f \cap \partial \Omega_f$, where $\partial \Omega_f$ denotes
  the boundary of $\Omega_f$. However, as a consequence of
  Lemma~\ref{lem:trCmk}, we have $\tr_f C_{m,f}^k u = \tr_f u$, and if
  $C_{m,f}^k u$ is smooth, we must also have $\tr_{\partial \Omega_f}
  C_{m,f}^k u = 0$, since $C_{m,f}^k u$ vanishes on the complement of
  $\Omega_f$. This apparent contradiction is exactly why the space of
  rational differential forms, $\0\Lambda_m^k(\T,f)$, appears as part
  of our construction. Furthermore, the statement $C_m^k u \in \Lambda^k(\T)$
  has the interpretation that there is a $w \in \Lambda^k(\T)$ such that 
  \[
  w_x = \sum_{\substack {f \in \Delta_{m+j}(\T)\\ 0 \le j \le k}} (C_{m,f}^k u)_x, 
\quad x \in \Omega \setminus \Delta_{m-1}, \quad 
\Delta_{m-1} = \bigcup_{g \in \Delta_{m-1}(\T)} g.
  \]
In particular, $C_n^k$ is the identity operator.
\end{remark}

Proposition~\ref{prop:preservation} also contains
the result that the operators $C_m^k$ commute with the exterior
derivative, i.e.,
\[
d C_m^k u = C_m^{k+1} du, \quad u \in \Lambda^k(\T), \, 0 \le k \le n-1.
\]
As a consequence of the properties of the cut--off operators $C_m^k$
just stated, we show that the operators $B_m^k$ preserve piecewise
smoothness, that they commute with the exterior derivative, that the
functions $B_{m,f}^ku$ are rational differential forms with local
support, and that these local bubbles define a decomposition of $u$.
 
 \begin{thm}\label{thm:trBmk}
Let $u \in
\Lambda^k(\T)$. Then we have 
\begin{itemize}
\item[i)] $B_m^ku \in \Lambda^k(\T), \quad 0 \le m\le n$,
\item[ii)] $d B_m^k u = B_m^{k+1} du, \quad 0 \le k \le n-1$,
\item[iii)] $B_{m,f}^k u \in \0\Lambda_{m+j}^k(\T ,f), \quad 
f \in \Delta_{m+j}(\T), \, 0 \le m \le n, \, 0 \le j \le k,$
\item[iv)]  $\tr_f \sum_{j=0}^{m}B_j^k u = \tr_f u, \quad 
f \in \Delta_m(\T), \, k \le m \le n,$
\end{itemize}
and the decomposition 
\[
u = \sum_{m=0} ^n B_m^k u = 
\sum_{m=0}^n \sum_{\substack{f \in \Delta_{m+j}(\T)\\ 0 \le j \le k}} B_{m,f}^k u.
\]
\end{thm}

\begin{proof}
  Property i) is a consequence of a simple induction argument, based
  on the iteration \eqref{Bmkdef}, and the corresponding property for
  the operator $C_m^k$ given in Lemma~\ref{lem:trCmk}.  The commuting
  property follows directly from \eqref{Bmkdef} and the corresponding
  property for the cut--off operators $C_m^k$, while property iii)
  follows from i), \eqref{def-Bm}, and the corresponding property for
  the operator $C_{m,f}^k$.  Furthermore, for $ f \in \Delta_m(\T)$,
  we have from \eqref{Bmkdef} and Lemma~\ref{lem:trCmk} that
\[
\tr_f B_m^k u = \tr_f (u - \sum_{j=0}^{m-1} B_j^k u), \quad k \le m \le n,
\]
and this implies property iv). Finally, the decomposition of $u$ is a
special case of property iv), corresponding to $m=n$.
\end{proof}
We emphasize that we do not claim that that each local operator
$B_{m,f}^k$ commutes with the exterior derivative. We have explained
above that we need to consider the global operator $B_m^k$ to
preserve piecewise smoothness, and in the same way we also need to
consider these global operators to obtain the commuting relation.

Recall that the spaces $\P_r \Lambda^k(\T)$ and $\P_r^- \Lambda^k(\T)$
are subspaces of $\Lambda^k(\T)$. More precisely, these spaces consist
of piecewise smooth differential forms which are polynomial forms of
class $\P_r$ or $\P_r^-$ on each $n$ simplex in
$\Delta_n(\T)$. Another key property of the bubble transform is that
the operators $B_m^k$ are invariant with respect to the piecewise
polynomial spaces, i.e., they map the spaces $\P_r \Lambda^k(\T)$ and
$\P_r^- \Lambda^k(\T)$ into themselves.  As above, this invariance
property is a consequence of a corresponding property for the cut--off
operators $C_m^k$.  In Proposition~\ref{prop:preservation}, it is established
that
\[
C_m^k (\P_r \Lambda^k(\T)) \subset \P_r \Lambda^k(\T), \quad \text{and } 
C_m^k (\P_r^- \Lambda^k(\T)) \subset \P_r^- \Lambda^k(\T).
\]
As a consequence, we obtain the following analogous result for the
operators $B_m^k$.
\begin{thm}
\label{thm:Bmk-preservation} The operators $B_m^k$ satisfy
\[
B_m^k (\P_r \Lambda^k(\T)) \subset \P_r \Lambda^k(\T) \quad \text{and} \quad 
B_m^k (\P_r^- \Lambda^k(\T)) \subset \P_r^- \Lambda^k(\T)
\]
for $0 \le m \le n$.
\end{thm}

\begin{proof}
  This follows directly from the iteration \eqref{Bmkdef} and the
  corresponding results for the operators $C_m^k$.
\end{proof}

Recall that up to now we have only considered the operators
$B_{m,f}^k$ and $B_m^k$ applied to functions in $\Lambda^k(\T)$, i.e.,
to piecewise smooth differential forms.  However, another desired
property of the bubble transform is that both the local operators
$B_{m,f}^k$ and the global operators $B_m^k$ are $L^2$ bounded
operators in the sense described in \eqref{B-bound}. In particular,
the constant $c$ appearing in \eqref{B-bound} only depends on the mesh
$\T$ through {\it the shape--regularity constant } $c_{\T}$ defined by
\begin{equation}\label{shape-constant}
c_{\T} = \max_{T \in \T} \frac{\diam(T)}{\diam(\Ball_T)},
\end{equation}
where $\Ball_T$ is the largest ball contained in $T$.  As a
consequence, since the space of piecewise smooth forms is dense in the
corresponding $L^2$ space, $L^2\Lambda^k(\Omega)$, we can conclude
that the operators $B_{m,f}^k$ and $B_m^k$ can be extended to bounded
linear operators defined on $L^2\Lambda^k(\Omega)$.  Furthermore, as
a consequence of the commuting property of the operator $B_m^k$, we can
also conclude that this operator is bounded in $H\Lambda^k(\Omega)$.  The
precise statements of the various bounds we will obtain will be given
in Section \ref{sec:bounds} below, cf. Theorems~\ref{thm:L2bound} and
\ref{thm:Hbound}.

\section{The case of scalar valued functions }
\label{sec:k=0}
The bubble transform for scalar valued functions, or zero--forms, was
introduced in \cite{bubble-I}.  In this section we will give a review
of some of the results from \cite{bubble-I}. It was established in
\cite{bubble-I} that the bubble transform for zero forms is an $L^2$
bounded map. However, in the present section, we will only discuss the
transform in the setting of piecewise smooth scalar valued functions,
i.e., for functions in $\Lambda^0(\T)$.

As we argued in Section~\ref{sec:Bmkprop} above, the main remaining
step to define the bubble transform for zero forms is to specify the
operators
\[
C_m^0 u = \sum_{f \in \Delta_m(\T)} C_{m,f}^0u,
\]
where each local operator $C_{m,f}^0$ maps piecewise smooth functions,
i.e., functions in $\Lambda^0(\T_f)$, to rational functions in the
space $\0\Lambda_m^0(\T,f)$.  The operator $C_{m,f}^0$ is defined by
\begin{equation*}
C_{m,f}^0u = \sum_{g \in \bar \Delta(f)} (-1)^{|f| - |g|} 
\frac{\rho_f}{\rho_g} L_{g}^* A_f^0u.
\end{equation*}
Here we recall that $\bar \Delta(f) = \Delta(f) \cup\{\emptyset
\}$, i.e., $g$ is allowed to be the empty set in the sum
above. The magic property of the operator $C_{m,f}^0$ is that it
preserves the trace of $u$ on $f$, but at the same time the function
$C_{m,f}^0u$ has support on $\Omega_f$.  In the simplest case, when
$m=0$, say $f = x_0$, the function $C_{m,f}^0u$ is given by
\[
(C_{m,f}^0u)_x =  (A_f ^0u)_{\lambda_0}  - (1 - \lambda_0)(A_f^0u)_0,
\]
while if $f = [x_0,x_1] \in \Delta_1(\T)$, then
\begin{multline*}
(C_{m,f}^0u)_x =  (A_f ^0u)_{\lambda_0, \lambda_1}  
- \frac{1 - \lambda_0 - \lambda_1}{1 - \lambda_0}(A_f^0u)_{\lambda_0,0}
\\
- \frac{1 - \lambda_0 - \lambda_1}{1 - \lambda_1}(A_f^0u)_{0,\lambda_1}
+ (1 - \lambda_0 - \lambda_1)(A_f^0u)_{0,0},
\end{multline*}
where in all cases $\lambda_i= \lambda_i(x)$.  The rational functions
$\rho_f/\rho_g$ for $g \in \Delta(f)$ will satisfy
\[
 \rho_f(x) \le \frac{\rho_f(x)}{\rho_g(x)} \le 1,
 \]
and if $g \neq f$ then $(\rho_f/\rho_g)|_f = 0$.
On the other hand, when $g =f$, then $\rho_f/\rho_g \equiv 1$.  
We therefore can conclude that
\[
\tr_f C_{m,f}^0 u  = \tr_f A_f^0 u = \tr_f u,
\]
where we have used part iii) of Lemma~\ref{lem:A-prop} for the final
equality.  To see that $C_{m,f}^0u$ has support on the macroelement
$\Omega_f$, we consider the function $C_{m,f}^0 u$ on the complement
of $\Omega_f$, i.e., where at least one function $\lambda_i$ for $i
\in I(f)$ vanishes. For simplicity, we can assume that $0 \in I(f)$ and
we consider a point $x \in \Omega$ such that $\lambda_0(x) = 0$.  For
any $g \in \bar \Delta(f)$ such that $x_0 \notin g$, let $g^\prime \in
\Delta(f)$ be such that $g^\prime \setminus g = x_0$.  Then, at the
point $x$, $\rho_{g^\prime}(x) = \rho_{g}(x)$ and $L_{g^\prime}(x) =
L_g(x)$, which implies that
\begin{equation}\label{cancel}
\Big[\frac{\rho_f}{\rho_g}L_g^*- 
\frac{\rho_f}{\rho_{g^\prime}}L_{g^\prime}^*\Big] A_f^0 u = 0
\end{equation}
at $x$.  By summing over all pairs $g$ and $g^\prime$, we can conclude
that $C_{m,f}^0 u =0$ at $x$, and hence it is identically zero on
$\Omega \setminus \Omega_f$. We summarize the results so far in the
following lemma.

\begin{lem}\label{lem:trace-preserve-C0}
Let $u \in \Lambda^0(\T_f)$ and $f \in \Delta_m(\T)$ for
$0 \le m \le n$. The  function
$C_{m,f}^0u$ satisfies $\tr_f C_{m,f}^0 u = \tr_f u$ and 
$C_{m,f}^0 u \equiv 0$ in $\Omega \setminus \Omega_f$.
\end{lem}

In general, the operator $C_{m,f}^0$ will not map piecewise smooth
functions to piecewise smooth functions, due to the singularity of the
rational functions $\rho_f/\rho_g$. On the other hand, if $g \in
\Delta(f), \, g \neq f$, then $g \subset \partial \Omega_f$. The
following result shows that if $u$ is piecewise smooth, with
$\tr_{\partial f} u = 0$, then $C_{m,f}^0 u$ is piecewise smooth. Furthermore, 
piecewise polynomials are preserved by the operator $C_{m,f}^0$ in this case.

\begin{lem}\label{lem:piec-smooth-C0}
  If $u \in \Lambda^0(\T_f)$ and $\tr_{\partial f} u =0$, then
  $C_{m,f}^0u \in \Lambda^0(\T)$. Furthermore, if in addition 
   $u \in \P_r\Lambda^0(\T_f)$, then  $C_{m,f}^0 u \in \P_r\Lambda^0(\T)$.
\end{lem}

\begin{proof}
  It follows from Lemma~\ref{lem:A-prop} that $A_f^0 u$ is a smooth
  function on $\S_f^c$.  Furthermore, since $L_f : f \to \S_f$ is an
  isomorphism, mapping $\partial f$ to $\partial S_f$, it follows from
  part iii) of Lemma~\ref{lem:A-prop} that $\tr_{\partial \S_f} A_f u =
  0$. In particular, for any $g \in \Delta(f)$, $g \neq f$, we have
  $\tr_{S_g} u = 0$. Since $\S_g$ has codimension one as a subset of
  $\S_g^c$, we can conclude that $\tr_{\S_g^c} b^{-1} A_f^0 u$ is a
  smooth function on $\S_g^c$, and as a consequence,
\[
L_g^* (b^{-1} A_f^0u) = \rho_g^{-1} L_g^* A_f^0 u
\]
is a smooth function on $\Omega$. Since this holds for all $g \in
\Delta(f)$, $ g \neq f$, we can conclude that $C_{m,f}^0 u$ is
piecewise smooth. In addition, if $u \in \P_r\Lambda^0(\T_f)$, then 
$\rho_g^{-1} L_g^* A_f^0 u \in \P_{r-1}\Lambda^0(\T_f)$ by part 
ii) of Lemma~\ref{lem:A-prop}, and hence $C_{m,f}^0 u \in \P_r\Lambda^0(\T)$.
\end{proof}

\begin{remark}
  The result given in Lemma~\ref{lem:piec-smooth-C0} was the key property
  used in \cite{bubble-I} to show that the bubble transform for zero forms
  preserves piecewise smoothness and piecewise polynomials.
  Surprisingly, this result will not play a corresponding role for the
  discussion given in this paper.  Instead, we will show below,
  cf. Section~\ref{sec:poly-pre}, that even if each individual
  operator $C_{m,f}^0$ maps piecewise smooth functions into rational
  functions, the complete operator, $C_m^0$, will indeed map both the
  space of piecewise smooth functions and piecewise polynomials into
  themselves.
\end{remark}

\section{The primal cut off operator}
\label{sec:primalop}
As a first attempt to define the bubble transform for $k$-forms, in
the case $k \ge 0$, we will  define local cut-off operators
$C_{m,f}^k$ given by
\begin{equation}\label{Cmf-k}
C_{m,f}^ku = \sum_{g \in \bar \Delta(f)} (-1)^{|f| - |g|} 
\frac{\rho_f}{\rho_g} L_{g}^* A_f^ku, \quad f \in \Delta_m(\T).
\end{equation}
This is basically the same operator as we use for zero forms, but
where we have replaced the average operator $A_f^0$ with the
corresponding operator $A_f^k$.  In fact, the discussion leading up to
Lemma~\ref{lem:trace-preserve-C0} is still true for the case of
$k$-forms.  More precisely, assume that $0 \in I(f)$ and consider a
subset $\Gamma$ of $\Omega$ such that $\lambda_0 \equiv 0$ on
$\Gamma$. If $g,g^\prime \in \bar \Delta(f)$ are related such that
$g^\prime \setminus g = x_0$ then $\rho_{g^\prime} = \rho_g$ and
$L_{g^\prime}^* = L_g^*$ on $\Gamma$. As a consequence, the
cancellation argument used above shows that $C_{m,f}^k u$ is supported
on $\Omega_f$.  Furthermore, it follows from Lemma~\ref{lem:A-prop}
that $A_f^ku \in \Lambda^k(\S_f^c)$ and that $\tr_f C_{m,f}^k u =
\tr_f u$ if $f \in \Delta_m(\T)$ for $k \le m \le n$.  We summarize
these results as follows.

\begin{lem}
\label{lem:trace-preserve-Ck-primal}
Let $u \in \Lambda^k(\T_f)$ and $f \in \Delta_m(\T)$ for $0 \le k,m
\le n$.  Then $C_{m,f}^ku \in \0\Lambda_m^k(\T,f)$ and $\tr_f C_{m,f}^k u =
  \tr_f u$ for $k \le m \le n$.
\end{lem}

It is also a
consequence of this lemma, and by following the path of arguments used
for zero forms above, that we can use the operators $C_{m,f}^k$ to
produce a decomposition of $u \in \Lambda^k(\T_f)$ into local
bubbles. However, in the present case, there seems to be no direct
analog of Lemma~\ref{lem:piec-smooth-C0}.
This is due to the fact that a vanishing
trace condition for $k$--forms with respect to a manifold of
codimension one, only controls the value of the form applied to
tangent vectors, while we have no control of the form when it is
applied to vectors normal to the manifold. As a consequence, from a
vanishing trace condition we cannot extract a linear factor as we did
in the proof of the lemma above.

Another key property we would like to have for the bubble transform is
that it should commute with the exterior derivative. However, an
identity like $dC_{m,f}^ku = C_{m,f}^{k+1}du$ will in general not be
true for the operator introduced above, even in the case $k = 0$. To
see this, let us compute $dC_{m,f}^k u$ when the operator $C_{m,f}^k$
is given by \eqref{Cmf-k}.  We have
\begin{align*}
dC_{m,f}^k u &= \sum_{g \in \bar \Delta(f)} (-1)^{|f| - |g|} 
\Big(\frac{\rho_f}{\rho_g} L_{g}^* A_f^{k+1} du 
+ d\Big(\frac{\rho_f}{\rho_g}\Big) \wedge L_{g}^* A_f^k u \Big)\\
&= C_{m,f}^{k+1}du + \sum_{g \in \bar \Delta(f)} (-1)^{|f| - |g|} 
d\Big(\frac{\rho_f}{\rho_g}\Big) \wedge L_{g}^* A_f^k u .
\end{align*}
To better understand the commutator $dC_{m,f}^k u - C_{m,f}^{k+1}du$,
we will use the fact that as long as $x$ is restricted to $\Omega_f$,
\[
\frac{\rho_f}{\rho_g} 
= \frac{\sum_{j \in I(f^*)}\lambda_j}
{\sum_{j \in I(f^*)}\lambda_j + \sum_{p \in I(f \cap g^*)}\lambda_p}.
\]
As a consequence,
\[
d \Big(\frac{\rho_f}{\rho_g}\Big) = \sum_{p \in I(f \cap g^*)} 
\sum_{j \in I(f^*)} \frac{\phi_{[x_p,x_j]}}{\rho_g^2},
\]
where $\phi_{[x_p,x_j]} \in \P_1^-\Lambda^1(\T)$ denotes the Whitney
form associated to the simplex $[x_p,x_j]$, i.e., $\phi_{[x_p,x_j]} =
\lambda_p d\lambda_j - \lambda_j d\lambda_p$.  Therefore, the
commutator $dC_{m,f}^k u - C_{m,f}^{k+1}du$ can be written as
\begin{equation}\label{commutator}
dC_{m,f}^k u - C_{m,f}^{k+1}du = \sum_{g \in \bar \Delta(f)} (-1)^{|f| - |g|}
\sum_{p \in I(f \cap g^*)} \sum_{j \in I(f^*)} \frac{\phi_{[x_p,x_j]}}{\rho_g^2}
 \wedge L_{g}^* A_f^k u
\end{equation}
for $x \in \Omega_f$. In fact, the identity \eqref{commutator} also
holds on the complement of $\Omega_f$.  To see this, observe that from
the properties of $C_{m,f}^k$ and $C_{m,f}^{k+1}$ derived above, we can
conclude that the left hand side of the identity is zero on the
complement of $\Omega_f$. To show that this is also true for the right
hand side, we will use a cancellation property similar to
\eqref{cancel}.  Consider a point $x \in \Omega$ where $\lambda_i(x) =
0$ for some $i \in I(f)$. Consider $g,g^\prime \in \bar \Delta(f)$
such that $g^\prime \setminus g = \{x_i \}$. When we sum the
contributions from these two simplexes on the right hand side of
\eqref{commutator} we obtain, up to a sign, 
\[
 \sum_{j \in I(f^*)} \Big[\sum_{p \in I(f \cap g^*)}\Big(\frac{\phi_{[x_p,x_j]}}
{\rho_g^2}\wedge L_{g}^* A_f^0 u
-\frac{\phi_{[x_p,x_j]}}{\rho_{g^\prime}^2}\wedge L_{g^\prime}^* A_f^0 u \Big)
 + \frac{\phi_{[x_i,x_j]}}{\rho_{g^\prime}^2}\wedge L_{g^\prime}^* A_f^0 u \Big].
\]
However, at points where $\lambda_i(x) = 0$, we have $\phi_{[x_i,x_j]}=
0$. Therefore, the last term can be dropped, and the rest of the terms
cancel when $\lambda_i(x) = 0$.  We can therefore
conclude that \eqref{commutator} holds in all of $\Omega$.

By summing the identity \eqref{commutator} over all $f \in
\Delta_m(\T)$, we obtain
\begin{equation*} 
\sum_{f \in \Delta_m(\T)}(dC_{m,f}^k u -  C_{m,f}^{k+1}du)
=\sum_{g \in \bar \Delta(\T)} 
\sum_{\substack{ f \in \Delta_m(\T)\\ f \supset g}}(-1)^{|f| - |g|}
\sum_{\substack{p \in I(f \cap g^*) \\j \in I(f^*)}} \frac{\phi_{[x_p,x_j]}}{\rho_g^2}
 \wedge L_{g}^* A_f^k u.
\end{equation*}
Note that if $g \in \bar \Delta$ is fixed, $f \in \Delta_m(\T)$, and
$x_p,x_j$ are such that $f \supset g$, $x_p \in f \cap g^*$, and $x_j
\in f^*$, then there is a unique element $f^\prime \in \Delta_m(\T)$ such
that $f \cap f^\prime \in \Delta_{m-1}(\T)$ and
\[
f^\prime \supset g, \quad 
x_p \in (f^{\prime})^*, \quad \text{and } x_j \in f^\prime \cap g^*.
\]
In other words, as compared to $f$, for the simplex $f^\prime$, the
role of the vertices $x_p$ and $x_j$ are reversed. Hence, for both the
choices $(f, p, j)$ and $(f^\prime, j, p)$ in the sum above, the
fraction $\phi_{[x_p,x_j]}/\rho_g^2$ will appear, but with different
signs. Furthermore, up to a possible reordering, we have that $[f \cap
f^\prime, x_p,x_j] \in \Delta_{m+1}(\T)$. By using this observation,
the sum above can be rewritten as
\begin{multline}\label{dCm-k}
\sum_{f \in \Delta_m(\T)}(dC_{m,f}^k u -  C_{m,f}^{k+1}du) 
\\
= \sum_{f \in \Delta_{m+1}(\T)} 
\sum_{e \in \Delta_1(f)} \sum_{g \in \bar \Delta(f \cap e^*)}
 (-1)^{|f| - |g|}\frac{\phi_{e}}{\rho_g^2}
 \wedge L_{g}^* (\delta A^ku)_{e,f}, 
\end{multline}
 where 
 \[
 (\delta A^ku)_{e,f}  
= \sum_{j \in I(e)} (-1)^{\sigma_e(x_j) }A_{f(\hat x_j)}^ku, \quad e \in \Delta_1(f).
 \]
 Here the hat notation is used to indicate that the vertex $x_j$
 should be removed from $f$, so that $f(\hat x_j) \in \Delta_m(\T)$, and
 we recall from Section~\ref{sec:prelims} above that $\sigma_e(x_j)$
 denotes the internal numbering of the vertex $x_j$ with respect to
 the simplex $e$.
 
 In order to obtain operators $C_m^k$ that commute with the exterior
 derivative, we have to include the contribution from the triple sum
 defining the commutator in the definition of these operators.  Recall
 that for $k= 0$, we have already defined the operator $C_m^0 =
 \sum_{f \in \Delta_m(\T)} C_{m,f}^0$.  Hence, for $k=0$, the identity
 \eqref{dCm-k} can be rewritten as
\begin{equation*}
 dC_m^0 u -  \sum_{f \in \Delta_m(\T)} C_{m,f}^{1}du 
= \sum_{f \in \Delta_{m+1}(\T)} 
\sum_{e \in \Delta_1(f)} \sum_{g \in \bar \Delta(f \cap e^*)}
 (-1)^{|f| - |g|}\frac{\phi_{e}}{\rho_g^2}
 \wedge L_{g}^* (\delta A^0u)_{e,f}.
 \end{equation*}
 It is easy to see that if $u$ is a constant scalar valued function,
 then $(\delta A^0u)_{e,f} = 0$, and as a consequence, we can conclude
 that $(\delta A^0u)_{e,f}$ only depends on $du$. Therefore, if for
 any $f \in \Delta(\T)$ and $e \in \Delta_1(f)$, we can construct operators
 $R_{e,f}^1$, mapping one forms to zero forms, such that
 \begin{equation}\label{R1d-delta}
 R_{e,f}^1 du = \tr_{\S_{f \cap e^*}^c} (\delta A^0u)_{e,f},
 \end{equation}
then the triple sum above can be expressed in terms of $du$. 

We summarize the discussion so far in the following lemma.
\begin{lem}\label{lem:commute-C1}
  Assume that for each $f \in \Delta(\T)$ and $e \in \Delta_1(f)$, we can
  construct operators $R_{e,f}^1$ such that relation \eqref{R1d-delta}
  holds. Define the operator $C_m^1$ by
 \begin{equation*}
 C_m^1 u = \sum_{\substack {f \in \Delta_{m+j}(\T)\\ 0 \le j \le 1}} C_{m,f}^1 u,
 \end{equation*}
 where $C_{m,f}^1$ is given by \eqref{Cmf-k} if $f \in \Delta_m(\T)$, and by 
 \begin{equation*}
 C_{m,f}^1u =  \sum_{e \in \Delta_1(f)} \sum_{g \in \bar \Delta(f \cap e^*)}
 (-1)^{|f| - |g|}\frac{\phi_{e}}{\rho_g^2}
 L_{g}^*R_{e,f}^1 u
 \end{equation*}
 if $f \in \Delta_{m+1}(\T)$. Then the commuting relation 
 $dC_m^0 u =
 C_m^{1}du$ holds.  
 \end{lem}
 We will delay the full analysis of the operator $C_m^1$
 until we have defined the operators $C_m^k$ in general. To do that,
 we will need a general class of \emph{order reduction operators},
 $R_{e,f}^k$, mapping $k$--forms to $(k- j)$--forms, where $f \in \Delta(\T)$ and $e \in \Delta_j(f)$. We will
 construct these operators, with the properties we will need, in the
 next section.

 \section{The order reduction operators $R_{e,f}^k$}
\label{sec:Refk}
Above we saw that in order to complete the definition of the operator
$C_m^1$, so that we obtain the commuting relation $dC_m^0 u= C_m^1
du$, we needed local operators $R_{e,f}^1$, where $f \in \Delta_{m+1}(\T)$
and $e \in \Delta_1(f)$, satisfying the identity \eqref{R1d-delta}. In
general, to construct the operators $C_m^k$, we will utilize a family of
local operators $R_{e,f}^k$, where $f \in \Delta(\T)$ and $e \in
\Delta_j(f), \, 0 \le j \le k$, which maps a $k$ form $u$ to a $k-j$
form $R_{e,f}^k u$. More precisely, for any $f \in \Delta(\T)$ and $e \in
\Delta_j(f)$, the operators $R_{e,f}^k$ belong to $\mathcal
L(\Lambda^k(\T),\Lambda^{k-j}(\S_{f \cap e^*}^c))$. In other
words, the linear operator $R_{e,f}^k$ maps piecewise smooth $k$ forms
defined on $\Omega$ to smooth $k-j$ forms defined on $\S_{f \cap
  e^*}^c$.  This section will be devoted to a general discussion of
these operators.

\subsection{The general pullback operator $G^*$}
\label{gen-pullback}
The function $G(y, \lambda) = \sum_{i \in \I} \lambda_ix_i + b(\lambda)y$, 
 mapping the product spaces $\Omega_f \times \S_f^c$ to $\Omega_f$,
 will play a key role in the construction.
The corresponding pullback, $G^*$, is a map
\[
G^* : \Lambda^k(\Omega_f) \to \Lambda^k(\Omega_f \times \S_f^c).
\]
We recall that a space of $k$--forms on a product space can be expressed
by the tensor product as
\[
\Lambda^k(\Omega_f \times \S_f^c) 
= \sum_{j=0}^k \Lambda^j(\Omega_f) \otimes \Lambda^{k-j}(\S_f^c).
\]
Here the symbol $\otimes$ is the tensor product.
In other words, elements 
$U \in \Lambda^j(\Omega_f) \otimes \Lambda^{k-j}(\S_f^c)$ 
can be written as a sum of terms of the form  
\[
a(y,\lambda) dy^j \otimes d\lambda^{k-j},
\]
where $dy^j$ and $d\lambda^{k-j}$ run over bases in $\Alt^j(\Omega_f)$
and $\Alt^{k-j}(\S_f^c)$, respectively, and where $a$ is a scalar
function on $\Omega_f \times \S_f^c$.  
Here $\Alt^k$ refers to  the corresponding space of algebraic $k$-forms.
Furthermore, for each $j$, $0
\le j \le k$, there is a canonical map $\Pi_j : \Lambda^k(\Omega_f
\times \S_f^c) \to \Lambda^j(\Omega_f) \otimes \Lambda^{k-j}(\S_f^c)$
such that
\[
U = \sum_{j=0}^k \Pi_j U, \quad U \in \Lambda^k(\Omega_f \times \S_f^c).
\]
The function $\Pi_j G^* u \in \Lambda^j(\Omega_f) \otimes
\Lambda^{k-j}(\S_f^c)$ can be identified as
\begin{multline*}
(\Pi_j G^* u)_{y,\lambda} ( t_1, \ldots t_j, v_1, \ldots ,v_{k-j})
\\
= u_{G(y,\lambda)}(D_{y}Gt_1, \ldots D_{y}Gt_j, D_{\lambda}Gv_1, 
\ldots ,D_{\lambda}Gv_{k-j}),
\end{multline*}
where  the tangent vectors $t_i \in T(\Omega_f)$ and $v_i \in T(\S_f^c)$.
For the special function $G$ in our case, 
\[
D_{y}G = b(\lambda) I, \quad \text{and } D_{\lambda} 
= \sum_{i \in \I} (x_i - y) d\lambda_i,
\]
so this can be rewritten as  
\[
(\Pi_j G^* u)_{y,\lambda} (t_1, \ldots t_j, v_1, \ldots ,v_{k-j})
=  b(\lambda)^j u_{G(y,\lambda)}(t_1, \ldots t_j, D_{\lambda}Gv_1, 
\ldots ,D_{\lambda}Gv_{k-j}).
\]
The basic commuting property for pull-backs, namely $d G^* = G^* d$, can be
expressed in the present setting as
\begin{equation}
\label{dG*}
d_{\Omega} \Pi_{j-1} G^*u  + (-1)^{j} d_S \Pi_j G^*u = \Pi_j d G^*u = 
\Pi_j G^* du, \quad j=1, \ldots, k,
\end{equation}
where $d_{\Omega}$ and $d_{S}$ denote the exterior derivative with
respect to the spaces $\Omega$ and $\S$, respectively.

We recall that in Section~\ref{sec:prelims} we introduced the average
operators $A_f^k$ for each $f \in \Delta$ mapping
$\Lambda^k(\Omega_f)$ to $\Lambda^k(\S_f^c)$.  There we defined the
operators $A_f^k$ from an integral with respect to $y$ of the
pullbacks $G(y, \cdot)^*$.  Alternatively, we can now identify these
operators as
\[
(A_f^k u)_{\lambda} = \aint_{\Omega_f} (\Pi_0 G^*u)_{\lambda} \wedge \vol, 
\quad \lambda \in \S_f^c,
\]
where $\vol$ is the volume form on $\Omega$.  For any $f \in \Delta$
and $e \in \Delta_j(f)$, the operators $R_{e,f}^k$ will be of the form
\begin{equation}
\label{genRdef}
(R_{e,f}^k u)_{\lambda}
=  \int_{\Omega} (\Pi_j G^*u)_{\lambda} \wedge z_{e,f}, \quad 
\lambda \in \S_{f \cap e^*}^c,
\end{equation}
where the weight function $z_{e,f}$ is an $n-j$ form on $\Omega$ with
local support.  This means that $R_{e,f}^k u$ is a $k-j$ form on
$\S_{f \cap e^*}^c$ for $0 \le j \le k$, while $R_{e,f}^k u \equiv 0$
for $j > k$.

For any $e \in \Delta_0(f)$, the operator $R_{e,f}^k = - \tr_{\S_{f
    \cap e^*}^c} A_f^k$, which corresponds to the operator
\eqref{genRdef}, where the $n$ form $z_{e,f}$ is given by
\begin{equation}\label{z-initial}
z_{e,f} = -\frac{\kappa_f}{|\Omega_f|} \vol \equiv - \vol_f,
\end{equation}
where $\kappa_f$ is the characteristic function of $\Omega_f$. In
other words, $\vol_f$ is the scaled version of the volume form
restricted to $\Omega_f$, such that
$\int_{\Omega_f} \vol_f = 1$.
To complete the definition of the operators $R_{e,f}^k$, we need to
specify the functions $z_{e,f}$ for $e \in \Delta_j(f)$ and $ j>0$.

\subsection{The weight functions $z_{e,f}$}
\label{weight-func}
For each $f \in \Delta$ and $e \in \Delta(f)$, the corresponding
functions $z_{e,f}$ will have support on a subdomain of $\Omega$
referred to as $\Omega_{e,f}$. The domains $\Omega_{e,f}$ can be
defined by a recursive process.  If $e$ is the emptyset or $e \in
\Delta_0(f)$, then $\Omega_{e,f}$ is taken to be $\Omega_f$.  For $j >
0$, we define the domains $\Omega_{e,f}$ recursively by
 \[
 \Omega_{e,f} = \bigcup_{i \in I(e)} \Omega_{e(\hat x_i),f(\hat x_i)}.
 \] 
An alternative characterization of the domains $\Omega_{e,f}$ is 
\[
\Omega_{e,f} = \Omega_{f \cap e^*} \cap \Omega_e^E,
\]
which can be verified by induction with respect to $|e|$. Here we
recall that the extended macroelements $\Omega_f^E$ are defined in
Section~\ref{sec:notation} above.  Note that this characterization
gives $\Omega_{\emptyset, f} = \Omega_{x_i,f} =\Omega_f$ and
$\Omega_{f,f} = \Omega_f^E$. As a consequence, if $e,g \in \Delta(f),
\, e \subset g$ and $i \in I(e)$ then
\begin{equation}\label{Omega-e-f}
\Omega_{e(\hat x_i),f} \subset \Omega_{e,f} \subset \Omega_{e,g} \subset \Omega_{g,g} 
 = \Omega_g^E \subset \Omega_f^E.
 \end{equation}
In particular, we observe that the $n$ simplexes forming
$\Omega_{e,f}$ are just a subset of the $n$ simplexes forming
$\Omega_{f \cap e^*}$. 


\eject
\null
\vskip 0.4in
\setlength{\unitlength}{0.5cm}
\begin{figure}[h]
\centering
\begin{picture}(5,0)
\put(0,0){\line(2,3){2}}
\put(0,0){\line(-2,3){2}}
\put(0,0){\line(1,0){4}}
\put(0,0){\line(2,-3){2}}
\put(3.5,1.3){$e$}
\put(1.8,.6){$f$}
\put(-.8,-.2){$x_0$}
\put(4.1,-.2){$x_1$}
\put(2.3,2.7){$x_2$}
\thicklines
\put(4,0){\line(-2,3){2}}
\put(4,0){\line(-2,-3){2}}
\put(-2,3){\line(1,0){4}}
\end{picture}
\vskip.6in
\caption{The domain $\Omega_{e,f}$ for $e = [x_1,x_2]$, $f=[x_0,x_1,x_2]$ 
and $n=2$.}
\label{fig:omega-ef}
\end{figure}
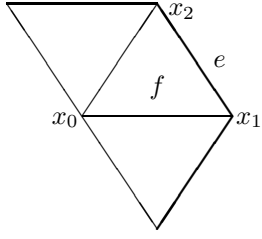
Recall that throughout the paper we have made the assumption that the
extended macroelements $\Omega_f^E$ are contractible. The following
result is an immediate consequence.
\begin{lem}
\label{lem:local-contractible}
 All the domains
  $\Omega_{e,f}$, for $f \in \Delta(\T)$ and $e \in \Delta(f)$, are
  contractible.
\end{lem}
\begin{proof}
  Since $\Omega_{f,f} = \Omega_f^E$ is assumed to be contractible, it
  is enough to consider the case $e \in \Delta(f)$, $e \neq f$.  But then 
  $f \cap e^*$ is nonempty. Since
  $\Omega_{e,f}$ is star-shaped with respect to any point in
  $f \cap e^*$ it follows that $\Omega_{e,f}$ is contractible.
\end{proof}
Since the domain $\Omega_{e,f}$ is contractible, it follows by de
Rham's theorem that the complex
$(\0 \P_1^- \Lambda^k(\T_{e,f}),d)$ is exact, e.g., see \cite[Section 5.5]{acta}
for further discussion of this fact. This property is crucial
for the construction which follows.
To define the functions $z_{e,f}$ for $e \in \Delta_j(f)$, $j >0$, we
 will introduce two difference operators defined for any set of
 functions parametrized by pairs $(e,f)$, $e \in \Delta(f)$ and $f
 \in \Delta(\T)$.  We define
\[
(\delta z)_{e,f} = \sum_{i \in I(e)} (-1)^{\sigma_e(x_i)} z_{e(\hat x_i), f(\hat x_i)}
\quad \text{and } \, (\delta^+ z)_{e,f} 
= \sum_{i \in I(e)} (-1)^{\sigma_e(x_i)} z_{e(\hat x_i), f}.
\]
It follows from standard arguments that these operators satisfy the
complex property $\delta^2 = 0$.  In fact, we have the following
identities.

\begin{lem}\label{lem:delta-prop}
The operators $\delta$ and $\delta^+$ satisfy
\[
\delta \circ \delta = 0, \quad \delta^+ \circ \delta^+ = 0, \quad \text{and } 
\delta \circ \delta^+ = - \delta^+ \circ \delta.
\]
\end{lem}

\begin{proof}
The two first properties are standard, so we omit the proofs.
To see the third identity, we compute the two expressions as
\[
(\delta^+ \delta z)_{e,f} 
= \sum_{\substack{ i,p \in I(e)\\ i \neq p}}(-1)^{\sigma_e(x_i) + \sigma_{e(\hat x_i)}(x_p)}
z_{e(\hat x_i,\hat x_p), f(\hat x_p)},
\]
and
\[
(\delta \delta^+ z)_{e,f} 
= \sum_{\substack{ i,p \in I(e)\\ i \neq p}}(-1)^{\sigma_e(x_p) + \sigma_{e(\hat x_p)}(x_i)}
z_{e(\hat x_i,\hat x_p), f(\hat x_p)}.
\]
However, we will always have 
\[
(-1)^{\sigma_e(x_p) + \sigma_{e(\hat x_p)}(x_i)} 
= - (-1)^{\sigma_e(x_i) + \sigma_{e(\hat x_i)}(x_p)},
\]
which implies the desired identity.
\end{proof}
Note that if $e = [x_0,x_1] \in \Delta_1(f)$, then 
\[
(\delta^+ z)_{e,f} = z_{x_1,f} - z_{x_0,f} = \vol_f - \vol_f = 0.
\]
We will define all the functions $z_{e,f}$ such that they satisfy
$\delta^+ z = 0$.  More precisely, if $e \in \Delta_j(f)$, then
$z_{e,f} \in \0\P_1^-\Lambda^{n-j}(\T_{e,f})$ and for $j>0$, we have
\begin{equation}\label{z-prop}
d z_{e,f} = (-1)^{j+1}(\delta z)_{e,f}, \quad \text{and } 
\, (\delta^+ z)_{e,f} = 0, \quad  f \in \Delta(\T), e \in\Delta_j(f).
\end{equation}
In fact, for $j >0$, the functions $z_{e,f}$ will be of the form
$z_{e,f} = (\delta^+w)_{e,f}$, where the functions $w_{e,f}$ are
defined for $f \in \Delta$ and $e \in \bar \Delta(f)$.  If $e =
\emptyset$, we define $w_{e,f} = -\vol_f$. For $e \in \Delta_j(f)$, $j
\ge 0$, the functions $w_{e,f}$ will be required to satisfy
\begin{equation}\label{dw}
dw_{e,f} =(-1)^{j}( (\delta - \delta^+)w)_{e,f}.
\end{equation}
In the special case $e = x_i \in \Delta_0(f)$, we will require $w_{x_i,f}
\in \0\P_1^-\Lambda^{n-1}(\T_{f(\hat x_i)})$ such that
\[
dw_{x_i,f} = ((\delta - \delta^+)w)_{x_i,f} 
= (w_{\emptyset,f(\hat x_i)} - w_{\emptyset, f} )= 
\vol_f - \vol_{f(\hat x_i)}.
\]
This is possible since the right hand side has mean value zero on
$\Omega_{f(\hat x_i)}$. 
In addition, we make the functions $w_{x_i,f}$
unique by the standard orthogonality condition with respect to
$d\0\P_1^-\Lambda^{n-2}(\T_{f(\hat x_i)})$.
It now follows by an inductive process, utilizing the exactness of the 
complexes of the form $(\0\P_1^-\Lambda(\T_{e,f}),d)$, 
that we can construct functions $w_{e,f} \in
\0\P_1^-\Lambda^{n-j-1}(\T_{e,f})$ for all $e \in \Delta_j(f)$, $ j
>0$, such that \eqref{dw} holds, and with support of $( (\delta -
\delta^+)w)_{e,f}$ in
\[
\bigcup_{i \in I(e)}[\Omega_{e(\hat x_i), f(\hat x_i)} \cup \Omega_{e(\hat x_i), f} ] 
= \bigcup_{i \in I(e)}[\Omega_{e(\hat x_i), f(\hat x_i)} = \Omega_{e,f}.
\]
To see this, just observe that Lemma~\ref{lem:delta-prop} implies that 
\[
d[( (\delta - \delta^+)w)_{e,f}] = ((\delta - \delta^+)dw)_{e,f}
= (-1)^{j+1}((\delta \delta^+  + \delta^+ \delta) w)_{e,f} =0.
\]
Furthermore, the functions $w_{e,f}$ are uniquely determined if we add 
the standard orthogonality condition
\begin{equation}\label{w-orth-cond}
\int_{\Omega_{e,f}} w_{e,f} \wedge \star dq = 0, 
\quad q \in  \0\P_1^-\Lambda^{n-j-2}(\T_{e,f}),
\end{equation}
where $\star$ is the Hodge star operator.  The functions $z_{e,f}$,
defined by $z_{e,f} = (\delta^+w)_{e,f}$, satisfy the following
properties.

\begin{lem}
\label{lem:z-prop-l}
  Assume that $f \in \Delta(\T)$ and $e \in \Delta_j(f)$. The functions
  $z_{e,f}$, defined above by $z_{e,f} = (\delta^+ w)_{e,f}$, belong
  to $\0\P_1^-\Lambda^{n-j}(\T_{e,f})$ and satisfy the two
  identities \eqref{z-prop}.
\end{lem}

\begin{proof} The support property follows from the support property
  of the functions $w_{e,f}$, while $\delta^+z = 0$ follows from the
  complex property of the operator $\delta^+$.  Finally, for $e \in
  \Delta_j(f)$, we have
\begin{multline*}
dz_{e,f} = d(\delta^+w)_{e,f} = (\delta^+ dw)_{e,f} = (-1)^{j} ((\delta^+ \circ \delta) w)_{e,f} 
\\
= (-1)^{j+1} ((\delta \circ \delta^+) w)_{e,f}
= (-1)^{j+1} (\delta z)_{e,f},
\end{multline*}
and this completes the proof.
\end{proof}
 
\subsection{Properties of the operators $R_{e,f}^k$}
Since $z_{e,f}$ has support on $\Omega_{e,f}$, $R_{e,f}^k u$ only depends on 
$u$ restricted to $\Omega_{e,f}$ and we can write \eqref{genRdef} as
\begin{equation*}
(R_{e,f}^k u)_{\lambda}  = \int_{\Omega} (\Pi_j G^*u)_{\lambda} \wedge z_{e,f}
=  \int_{\Omega_{e,f}} (\Pi_j G^*u)_{\lambda} \wedge z_{e,f}, 
\quad \lambda \in \S_{f \cap e^*}^c.
\end{equation*}
For any $f \in \Delta(\T)$ and $e \in \Delta_j(f)$, we define 
for $\lambda \in \S_{f \cap e^*}^c$,
 \begin{equation*}
(\delta R^ku)_{e,f} = \sum_{i \in I(e)} (-1)^{\sigma_e(x_i)} 
R_{e(\hat x_i), f(\hat x_i)}^k u.
\end{equation*}
Note that for each $i \in I(e)$, we have $f(\hat x_i) \cap e(\hat
x_i)^* = f \cap e^*$. Therefore, $(\delta R^k u)_{e,f}$ is a $k-j+1$
form on $\S_{f \cap e^*}$.  Alternatively, we can represent $(\delta
R^ku)_{e,f}$ by
\begin{equation}\label{rep-delta-R}
((\delta R^ku)_{e,f})_{\lambda} = \int_{\Omega}
 (\Pi_{j-1} G^*u)_{\lambda} \wedge (\delta z)_{e,f}, \quad \lambda \in \S_{f \cap e^*}.
 \end{equation}
 We show below that the operators $R_{e,f}^k$ satisfy the relation
 \begin{equation}\label{Rkd-delta}
 R_{e,f}^{k+1} du = (-1)^j dR_{e,f}^k u - (\delta R^ku)_{e,f}, 
\quad e \in \Delta_j(f), \, 0 \le j \le k+1.
 \end{equation}
 We note that all the three terms appearing here are $k-j+1$ forms
 defined on the simplex $\S_{f \cap e^*}^c$, and that the desired
 formula \eqref{R1d-delta} is just a special case corresponding to
 $k=0$ and $j=1$.  Furthermore, if we define $R_{e,f}^k$ to be the
 zero operator when $e$ is the emptyset, then \eqref{Rkd-delta} with
 $j = 0$ expresses the commuting property of the operators $A_f^k$.
 In addition, we show below that the operators $R_{e,f}^k$ satisfy the identity
\begin{equation}\label{Rk-delta+}
(\delta^+ R^ku )_{e,f} = 0, 
\end{equation}
where 
\[
(\delta^+ R^ku)_{e,f} = \sum_{i \in I(e)} (-1)^{\sigma_e(x_i)} 
\tr_{\S_{f \cap e^*}^c} R_{e(\hat x_i), f}^k u.
\]
The identities \eqref{Rkd-delta} and \eqref{Rk-delta+} will be key
tools for constructing commuting cut--off operators $C_m^k$.  To
derive the identity \eqref{Rkd-delta}, we will use 
the basic commuting property for pull-backs, $dG^* = G^*d$, which in
the present setting is given by \eqref{dG*}, where
$\Pi_j G^* u \in \Lambda^j(\Omega_f) \otimes
\Lambda^{k-j}(\S_f^c)$, and the operators $d_{\Omega}$ and $d_{S}$
denote the exterior derivatives with respect to the spaces $\Omega$ and
$\S$, respectively.
\begin{prop}\label{prop:R-relation}
  The operators $R_{e,f}^k$ satisfy the two identities
  \eqref{Rkd-delta} and \eqref{Rk-delta+}.
\end{prop}

\begin{proof} 
For any $e \in \Delta_j(f)$, we have 
 \[
 (\delta^+ R^ku)_{e,f} =  \tr_{\S_{f \cap e^*}^c}\int_{\Omega}
 (\Pi_{j-1} G^*u) \wedge (\delta^+ z)_{e,f},
 \]
 and the relation \eqref{Rk-delta+} follows directly from the second
 identity of \eqref{z-prop}.  To show \eqref{Rkd-delta}, we use the
 first relation of \eqref{z-prop}, \eqref{rep-delta-R}, and
 integration by parts to obtain
\begin{multline*}
-(\delta R^{k} u)_{e,f }
= (-1)^{j} \int_{\Omega} \Pi_{j-1} G^* u \wedge d_{\Omega }z_{e,f}
\\
= \int_{\Omega} d_{\Omega }\Pi_{j -1}G^* u \wedge z_{e,f}
= \int_{\Omega} \Big[(-1)^{j+1} d_{\S }\Pi_{j}G^* u 
+ \Pi_jG^*du\Big] \wedge z_{e,f},
\end{multline*}
where we have used \eqref{dG*} to obtain the last equality. However, since
\begin{equation*}
d R_{e,f}^k u = \int_{\Omega} d_S \Pi_j G^* u \wedge z_{e,f},
\end{equation*}
we see that the right hand side above is exactly equal to
\[
(-1)^{j+1} dR_{e,f}^k u + R_{e,f}^{k+1} du ,
\]
and hence the desired result is obtained.
\end{proof}

We end this section by establishing the polynomial preservation
properties of the operators $R_{e,f}^k$.  We also show that the
operators $R_{e,f}$ map piecewise smooth differential forms to smooth
differential forms. In fact, the proposition below can be seen as a
generalization of Lemma~\ref{lem:A-prop}, and the two proofs are closely
related.

\begin{prop}\label{prop:R-pol-prop}
Assume that $f \in \Delta(\T)$, $e \in \Delta_j(f)$.
\begin{itemize}
\item[i) ]If $u \in \Lambda^k(\T)$, then $b^{-j}R_{e,f}^k u 
\in \Lambda^{k-j}(\S_{f \cap e^*}^c)$,
\item[ii)] If 
$u \in \P_r\Lambda^k(\T)$ then $b^{-j}R_{e,f}^k u 
\in \P_{r}\Lambda^{k-j}(\S_{f \cap e^*}^c)$,
\item[iii)] if $u \in \P_r^-\Lambda^k(\T)$ then 
$b^{-j} R_{e,f}^k u \in \P_{r}^-\Lambda^{k-j}(\S_{f \cap e^*}^c)$.
\end{itemize}
\end{prop}
\begin{proof}
  If $e =f$, then $\S_{f \cap e^*}^c$ consists of a single point, the
  origin in $\R^{N+1}$, and in this case the conclusion of the
  proposition is obvious.  Therefore, for the rest of the proof, we
  assume that $e \neq f$, such that $f \cap e^*$ is nonempty.  We
  recall that for $f \in \Delta(\T)$ and $e \in \Delta_j(f)$, we have
\[
R_{e,f}^k u =  \int_{\Omega_{e,f}} \Pi_j G^*u  \wedge z_{e,f}.
\]
More precisely, $R_{e,f}^ku$ is $k-j$ form on $S_{f \cap e^*}^c$ such that 
\[
(R_{e,f}^k u)_{\lambda}(v_1, \ldots , v_{k-j})  = \int_{\Omega_{e,f}} (\Pi_j G^* u  \lrcorner v_1 \ldots \lrcorner  v_{k-j})_{\lambda} \wedge z_{e,f},
\]
where $v_i \in T(\S_{f\cap e^*}^c)$ and $(\Pi_j G^* u \lrcorner v_1,
\ldots \lrcorner v_{k-j})_{\lambda}$ is a $j$ form on $\Omega$. In
fact,
\begin{multline}\label{integrand}
b(\lambda)^{-j} ((\Pi_j G^* u  \lrcorner v_1 \ldots \lrcorner  v_{k-j})_{\lambda})_y(t_1, \ldots ,t_j) \\
= u_{G(y,\lambda)}(t_1, \ldots ,t_j, D_{\lambda}Gv_1, 
\ldots ,D_{\lambda}Gv_{k-j}),
\end{multline}
where $y \in \Omega_{e,f}$ and 
$t_i \in T(\Omega_{e,f})$. 
Furthermore, for any fixed $y \in \Omega_{e,f} \subset \Omega_{f \cap
  e^*}$, the set
\[
\{ \, G(y, \lambda) \, : \, \lambda \in \S_{f\cap e^*}^c \, \}
\]
belongs to a single $n$ simplex of $\Omega_{e,f}$, while the vectors
of the form $D_{\lambda} v$ are independent of $\lambda$. This shows
that if $u$ is a piecewise smooth $k$ form on $\Omega_{e,f}$, then for
each fixed $y \in \Omega_{e,f}$, the right hand side of \eqref{integrand} is a
smooth function of $\lambda \in \S_{f \cap e^*}^c$.  The same must be
true for the integral with respect to $y$, and hence the first
statement of the proposition is established.

The second property follows from almost the same argument, since if
$u$ is a piecewise polynomial, i.e., $u \in \P_r\Lambda^k(\T)$, then
the right hand side of \eqref{integrand} is a polynomial of degree $r$
with respect to $\lambda \in \S_{f \cap e^*}^c$ for each fixed $y \in
\Omega_{e,f}$. Again, the same will hold for the integral with respect
to $y$.  To show that the $\P_r^-$ spaces are also preserved, we will
consider $R_{e,f}^k u \lrcorner \lambda$, where $\lambda \in \S_{f
  \cap e^*}^c$.  Then
\[
(R_{e,f}^k u)_{\lambda}(\lambda, v_1, \ldots , v_{k-j-1})  
= \int_{\Omega_{e,f}} (\Pi_j G^* u  \lrcorner \lambda \lrcorner v_1 \ldots 
\lrcorner  v_{k-j-1})_{\lambda} \wedge z_{e,f},  
 \]
where 
\begin{multline*}
b(\lambda)^{-j} ((\Pi_j G^* u  \lrcorner \lambda \lrcorner v_1 \ldots \lrcorner  v_{k-j-1})_{\lambda})_y(t_1, \ldots ,t_j) \\
= u_{G(y,\lambda)}(t_1, \ldots ,t_j, G(y,\lambda) - y, D_{\lambda}Gv_1, 
\ldots ,D_{\lambda}Gv_{k-j-1}).
\end{multline*}
However, if $u \in \P_r^-\Lambda^k(\T)$, it follows from the linearity of $G$
with respect to $\lambda$ that for each fixed $y \in \Omega_{e,f}$, the
right hand side above is in $\P_r(\S_{f \cap e^*}^c)$, and therefore the same holds for
the integral with respect to $y$. As a consequence, we can conclude
that $b^{-j}R_{e,f}^k u \in \P_{r}^-\Lambda^{k-j}(\S_{f \cap e^*}^c)$.
This completes the proof of the proposition.
\end{proof}

\section{The cut off operators $C_{m,f}^k$}
\label{Cmkops}
Recall that relation \eqref{R1d-delta} is just a special case of
\eqref{Rkd-delta}.  As a consequence of the construction of the order
reduction operators $R_{e,f}^k$ in the previous section, we therefore
can conclude that the operator $C_m^1$, specified in Lemma~\ref{lem:commute-C1},
satisfies the commuting relation $dC_m^0 = C_m^1 d$.

In general, for $0 \le k \le n$, we define the operator
$C_m^k$ by
\[
 C_m^k u = \sum_{\substack {f \in \Delta_{m+j}(\T)\\ 0 \le j \le k}} C_{m,f}^k u,
 \]
where $C_{m,f}^k$ is given by \eqref{Cmf-k} if $f \in \Delta_m(\T)$, and by 
 \begin{equation}
\label{C(m+j)f-k}
 C_{m,f}^ku =  j! \sum_{e \in \Delta_j(f)} \sum_{g \in \bar \Delta(f \cap e^*)}
 (-1)^{|f| - |g|}\frac{\phi_{e}}{\rho_g}
 \wedge L_g^*b^{-j}R_{e,f}^k u
 \end{equation}
 if $f \in \Delta_{m+j}(\T)$, $1 \le j \le k$.  Here we recall that
 $\phi_e$ is the Whitney form associated to the simplex $e$ and that
 $\rho_g = L_g^* b$.  We now have the following extension of
 Lemma~\ref{lem:trace-preserve-Ck-primal}.

\begin{lem}\label{lem:support-Ck}
  Let $u \in \Lambda^k(\T_f)$ and $f \in \Delta_{m+j}(\T)$ for $0\le m
  \le n$ and $0 \le j \le k$.  
Then  $C_{m,f}^ku \in \0\Lambda_m^k(\T ,f)$
and
$\tr_f C_m^k u = \tr_f u$ for $f \in \Delta_m(\T)$ and $k \le m \le n$.
\end{lem}
\begin{proof}
We only have to consider the case $j > 0$, since the case $j=0$ is covered by 
Lemma~\ref{lem:trace-preserve-Ck-primal}.
  Let $f \in \Delta_{m+j}(\T)$, $1 \le j \le k$, be fixed. It is enough
  to consider each term in the sum of $C_{m,f}^k u$ corresponding to
  $e \in \Delta_j(f)$ fixed, i.e.,
 \begin{equation*}
 C_{m,e,f}^k u := \sum_{g \in \bar \Delta(f \cap e^*)}
 (-1)^{|f| - |g|}\frac{\phi_{e}}{\rho_g}
 \wedge L_g^*b^{-j}R_{e,f}^k u.
\end{equation*}
By part (i) of Proposition~\ref{prop:R-pol-prop}, $b^{-j}R_{e,f}^k u
\in \Lambda^{k-j}(\S_{f \cap e^*}^c)$. As a consequence, it follows
that $C_{m,e,f}^k u \in \Lambda_m^k(\T ,f)$. To show that $C_{m,e,f}^k
u $ is supported in $\Omega_f$ we will use a variant of the
cancellation argument we have used before.  Assume that $i \in I(f)$
and let $\Gamma$ be a subset of $\Omega$ such that $\lambda_i \equiv 0
$ on $\Gamma$.  If $ i \in I(e)$, then $\phi_e$ vanishes on
$\Gamma$. On the other hand, if $i \notin I(e)$, then $i \in I(f \cap
e^*)$ and we can use a cancellation argument to show that $C_{m,e,f}^k
u = 0$.  We compare two terms in the definition of $C_{m,e,f}^k u$
corresponding to $g$ and $g^\prime$, where $g \subset g^\prime$ and
$g^\prime \setminus g = \{x_i \}$.  The two terms will cancel on
$\Gamma$.  Therefore we can conclude that $C_{m,e,f}^k u = 0$ on
$\Gamma$, and this implies the support property of $C_{m,e,f}^k u$.
To check the trace property of $C_m^k u$, we recall that if $g \in
\Delta_m(\T)$ and $f \in \Delta_{m+j}(\T)$, $j \ge 0$, where $f \neq
g$, then $g$ will not belong to the interior of $\Omega_f$.  By
combining this observation, the result above, and the trace property
given in Lemma~\ref{lem:trace-preserve-Ck-primal}, we can conclude
that $\tr_f C_m^k u = \tr_f u$ for $f \in \Delta_m(\T)$ and $m \ge k$.
\end{proof}

Next we will perform a modest rewriting of the operator $C_m^k u$
which will be useful in the discussion of the next section. 
We will split the operator $C_{m,f}^k$ for $f \in \Delta_m(\T)$ into
two terms. For $f \in \Delta_m(\T)$ and $g \in \bar \Delta(f)$, we have
\begin{equation*}
\frac{\rho_f}{\rho_g} L_{g}^* A_f^ku = 
(1 + \frac{\rho_f - \rho_g}{\rho_g})L_{g}^* A_f^ku 
= L_g^* A_f^k u + 
\sum_{e \in \Delta_0(f \cap g^*)}
\frac{\phi_e}{\rho_g}  \wedge L_{g}^* R_{e,f}^ku,
\end{equation*}
where we recall that $\phi_e = \lambda_i$ and 
$R_{e,f}^k = - A_f^k$ for $e = [x_i] \in \Delta_0(f)$. As a
consequence, the operator $C_m^k$ can be rewritten as
\begin{multline}\label{Cm-rewritten}
C_m^k u = \sum_{f \in \Delta_m(\T)} 
\sum_{g \in \bar \Delta(f)}(-1)^{|f| -|g|}L_g^* A_f^k u 
\\
+  \sum_{\substack{f \in \Delta_{m+j}(\T)\\ 0 \le j \le k}} j!  
\sum_{g \in \bar \Delta(f)}(-1)^{|f| -|g|}\sum_{e \in \Delta_j(f \cap g^*)}
 \frac{\phi_e}{\rho_g} \wedge L_{g}^* b^{-j} R_{e,f}^ku.
\end{multline}
In other words, we have written the operator $C_{m,f}^k$, for $f \in
\Delta_m(\T)$, as a sum of two operators, where both terms have
support on $\Omega_f$, and where the term containing $\phi_e$ for $e
\in \Delta_0(f)$ has the same form as the terms containing $\phi_e$,
for $|e| > 1$.

Recall that the operator $L_g^*$ maps smooth differential forms to
piecewise smooth forms, and that the operators $b^{-j}R_{e,f}^k$ for
$e \in \Delta_j(f)$ map piecewise smooth forms to smooth forms,
cf. Proposition~\ref{prop:R-pol-prop}.  Hence, it appears that all the terms
in the second part of \eqref{Cm-rewritten} contain a rational factor
$1/\rho_g$. The challenge is to show that this rational factor
disappears when we add the terms in the second part of
\eqref{Cm-rewritten}. This will be a consequence of the
discussion given in the next section.

\section{Properties of the global operators $C_m^k$}
\label{sec:poly-pre}
It will be a consequence of the result of this section that the
operator $C_m^k$ commutes with the exterior derivative. Furthermore,
we will show that this operator is invariant with respect to the
piecewise smooth space $\Lambda^k(\T)$, and with respect to the
piecewise polynomial spaces $\P_r\Lambda^k(\T)$ and
$\P_r^-\Lambda^k(\T)$. In other words, the operator $C_m^k$ maps these
spaces into themselves. In the special case when $m=n$, the operator
$\C_m^k $ reduces to the identity, which obviously has the desired
properties. Therefore, in the rest of the discussion of this section,
we can assume that $0 \le m \le n-1$.

We start by recalling the support properties of the operators
$C_{m,f}^k$ given in Lemma~\ref{lem:support-Ck}. It follows from the fact
that $C_{m,f}^ku$ has support on $\Omega_f$ that for each $n$ simplex
$T$ in $\Delta_n(\T)$, we have
\begin{equation}\label{C-restrict}
\tr_T C_m^k u = \sum_{\substack {f \in \Delta_{m+j}(T)\\ 0 \le j \le k}} C_{m,f}^k u,
\end{equation}
i.e., we can restrict the sum to the subsimplexes $f$ in $\Delta(T)$. 
Furthermore, if $T_-$ and $T_+$ are
two $n$ simplexes with a common $n-1$ simplex $T_- \cap T_+ \in
\Delta_{n-1}(\T)$, then
 \[ 
 \tr_{T_- \cap T_+} \tr_{T_-} C_m^k u = \tr_{T_- \cap T_+} \tr_{T_+} C_m^k u
 = \tr_{T_- \cap T_+} 
 \sum_{\substack{f \in \Delta_{m+j}(T_- \cap T_+)\\0 \le j \le k}} C_{m,f}^k u.
 \]
 This means that for any $u \in \Lambda^k(\T)$, the differential form
 $C_m^k u$ will always have single valued traces on all elements of
 $\Delta_{n-1}(\T)$.  As a consequence, to show that the operator
 $C_m^k$ is invariant with respect to the piecewise smooth space
 $\Lambda^k(\T)$ and the piecewise polynomial spaces, it is enough to
 consider the restriction of $C_m^k u$ to a single $n$ simplex $T$,
 where the restriction is given by \eqref{C-restrict}

\subsection{Restricting to a single $n$ simplex}
We will consider the restriction of  $C_m^k u$ to a fixed $n$ simplex $T$.
In fact, in the arguments given below, we can consider the part of
$\tr_T C_m^k u$ which corresponds to a fixed simplex $g
\in \bar \Delta(T)$. Therefore, for each fixed $T \in \Delta_n(\T)$
and $g \in \bar \Delta(T)$, $0 \le |g| \le m+1$, we introduce the
operator
\[
C_m^k(g,T) u =   \sum_{\substack{f \in \Delta_m(T)\\ f \supset g}} 
L_g^* A_f^k u 
+ \sum_{j=0}^k (-1)^j j! \sum_{\substack{f \in \Delta_{m+j}(T)\\ f \supset g}}
\sum_{e \in \Delta_j(f \cap g^*)}
 \frac{\phi_e}{\rho_g} \wedge L_{g}^* b^{-j}R_{e,f}^ku.
 \]
 If $u \in \Lambda^k(\T)$, we will view the function
 $C_m^k(g,T)u$ as a, possibly rational, $k$ form on $T$.
 It is a consequence of the characterization of $\tr_T C_m^k$, given
 by \eqref{C-restrict}, that
 \[
 \tr_T C_m^k u = \sum_{g \in \bar \Delta(T)} (-1)^{m+1 -|g|} C_m^k(g,T)u, 
\quad T \in \Delta_n(\T).
 \] 
 If we can show that each operator $C_m^k(g,T)$ commutes with the
 exterior derivative, and that it maps the spaces $\Lambda^k(\T)$,
 $\P_r\Lambda^k(\T)$, and $\P_r^-\Lambda^k(\T)$ into the corresponding
 spaces on $T$, then we can immediately conclude the following
 fundamental result.
\begin{prop}
\label{prop:preservation}
The operator $C_m^k $ satisfies the commuting relation
\[
d C_m^k = C_m^{k+1}d, \quad 0 \le k \le n-1.
\]
Furthermore, 
\begin{itemize}
\item[ i)] if $u \in \Lambda^k(\T)$, then $C_m^k u \in \Lambda^k(\T)$,
\item[ii)] if  $u \in \P_r \Lambda^k(\T)$, then
$C_m^k u \in \P_r \Lambda^k(\T)$,
\item[iii)] if $u \in \P_r^- \Lambda^k(\T)$, 
then $C_m^k u \in \P_r^- \Lambda^k(\T)$.
\end{itemize}
\end{prop}

The desired properties of the operator $C_m^k(g,T)$ will follow from
the following decomposition.
\begin{lem}
\label{lem:Qrewrite}
If $u \in \Lambda^k(\T)$ and $g \in \Delta_s(T)$, then
\begin{equation}
\label{Cmgk-ident}
C_m^k(g,T)u - \frac{n-m}{n-s} 
\sum_{\substack{f \in \Delta_m(T) \\ f \supset g}} L_g^* A_{f}^k u
= d Q_m^k u + Q_{m}^{k+1} du,
\end{equation}
where the operators $Q_m^k = Q_m^k(g,T)$ are given by
\begin{equation*}
Q_m^k u = \frac{1}{n-s}
\sum_{j = 1}^k (-1)^{j }(j-1)! \sum_{\substack{f \in \Delta_{m+j}(T)\\ f \supset g}} 
\sum_{e \in \Delta_j(f \cap g^*)} (\delta \phi)_{e}
\wedge L_g^* b^{-j}R_{e,f}^k u,
\end{equation*}
with
$(\delta \phi)_e = \sum_{i \in I(e)} (-1)^{\sigma(x_i)} \phi_{e(\hat x_i)}$.
In particular, $Q_m^0 = 0$ and the case $g =\emptyset$ corresponds to $s = -1$.
\end{lem}

We will delay the proof of this lemma, and first show how this
decomposition immediately leads to a proof of
Proposition~\ref{prop:preservation}.
\begin{proof} (of Proposition~\ref{prop:preservation}) Recall that we only
  need to consider $C_m^k(g,T) $ as an operator from $\Lambda^k(\T)$
  to the space of rational $k$ forms on $T$.  Since the 
    operator $A_f^k$ commutes with $d$, the commuting property will
  follow if the right hand side of \eqref{Cmgk-ident} commutes with
  $d$. However, this follows since
  \[
  d [dQ_m^k + Q_m^{k+1}d] u = d Q_m^{k+1}d u =  [dQ_m^{k+1} + Q_m^{k+2}d] du.
  \]
  From the properties of the operator $R_{e,f}^k$ given in
  Proposition~\ref{prop:R-pol-prop}, we can conclude that the operator
  $Q^k$ maps the space $\Lambda^k(\T)$ to $\Lambda^{k-1}(T)$ and
  $\P_r \Lambda^k(\T)$ to $\P_{r+1}\Lambda^{k-1}(T)$. The desired
  conclusion, that the operator $C_m^k$ maps the spaces
  $\Lambda^k(\T)$ and $\P_r\Lambda^k(\T)$ into themselves, follows
  directly from the decomposition \eqref{Cmgk-ident}. To show the
  corresponding result for the $\P_r^-$ spaces, we need to show that
  the operator $C_m^k(g,T)$ preserves these spaces.  However, it
  follows from the definition of the operator $C_m^k(g,T)$,
  Proposition~\ref{prop:R-pol-prop}, and formula (3.16) of \cite{acta}
  that $C_m^k(g,T)$ maps $\P_r^-\Lambda^k(\T)$ into
  $\rho_g^{-1} \P_{r+1}^-\Lambda^k(T)$. Since $\P_r^-\Lambda^k(\T)$ is
  a subspace of $\P_r\Lambda^k(\T)$ we therefore have that $C_m^k(g,T)$
  maps $\P_r^-\Lambda^k(\T)$ into
  \[
  \P_r\Lambda^k(T) \cap \rho_g^{-1} \P_{r+1}^-\Lambda^k(T).
  \]
  But elements of this space must be in $ \P_r^-\Lambda^k(T)$.  To see
  this, let $u \in \P_{r+1}^-\Lambda^k(T)$ be such that
  $\rho_g^{-1}u \in \P_{r} \Lambda^{k}(T)$.  For any
  $x_j \in \Delta_0(T)$, we then have
\begin{equation*}
u \lrcorner (x-x_j) \in \P_{r+1}\Lambda^{k-1}(T), 
\quad \text{and }\rho_g^{-1} (u \lrcorner (x-x_j)) \in \P_{r+1}\Lambda^{k-1}(T).
\end{equation*}
In other words, the polynomial form $u \lrcorner (x-x_j)$ has
$\rho_g$ as a linear factor, and as a consequence,
$\rho_g^{-1} (u \lrcorner (x-x_j))$ must be in $\P_{r}\Lambda^{k-1}(T)$. 
This implies that $\rho_g^{-1}u \in \P_{r}^- \Lambda^{k}(T)$.
\end{proof}

%

Before we prove Lemma~\ref{lem:Qrewrite}, we will first establish a
preliminary result.  To simplify the notation in the present setting,
where $T$ and $g$ are fixed, we introduce the set $\Delta(m,j)$ given
by
\[
\Delta(m,j) = \{ \,  (e,f) \, : \, f \in \Delta_{m+j}(T), \, f \supset g, 
\, e \in \Delta_j(f \cap g^*) \, \}.
\]
Furthermore, in the discussion  below, we abbreviate $g^*(T)$ by $g^*$.

We also introduce the operators $C_{m,\ell}^k(g,T)$ given by 
\[
C_{m,\ell}^k(g,T) u =   \sum_{\substack{f  \in \Delta_m(T)\\ f \supset g}} 
L_g^* A_f^k u 
+ \sum_{j=0}^{\ell} (-1)^j j! \sum_{(e,f) \in \Delta(m,j)}
 \frac{\phi_e}{\rho_g} \wedge L_{g}^* b^{-j}R_{e,f}^ku.
 \]
 We note that we have $C_{m,k}^k(g,T) = C_{m}^k(g,T)$, while the
 operator $C_{m,0}^k(g,T)$ corresponds to the primal cut off operator
 studied in Section~\ref{sec:primalop}, 
but rewritten as in \eqref{Cm-rewritten}.

\begin{lem}
\label{lem:Qrewrite-0}
If $g \in \Delta_s(T)$, then
\[
C_{m,0}^k(g,T)u - \frac{n-m}{n-s} 
\sum_{\substack{f \in \Delta_m(T) \\ f \supset g}} L_g^* A_{f}^k u
=  \frac{1}{n-s} \sum_{(f,e) \in \Delta(m,1)}
\frac{(\delta \phi)_{e}}{\rho_g}
\wedge L_g^* (\delta R^ku)_{e,f},
\]
where the case $g =\emptyset$ corresponds to $s = -1$.
\end{lem}

\begin{proof}
 We observe 
 that the desired identity will follow if we can show that 
\begin{multline}
\label{init-id}
 \sum_{(e,f) \in \Delta(m,1)}(\delta \phi)_{e}
\wedge L_g^* (\delta R^ku)_{e,f} - 
(n-s)  \sum_{(e,f) \in \Delta(m,0)}
 \phi_e \wedge L_{g}^* R_{e,f}^ku
\\
=(m-s) \rho_g \sum_{\substack{f \in \Delta_{m}(T)\\ f \supset g}} L_g^*A_f^k u.
\end{multline}
By using the special definitions of $\phi_e$ and $R_{e,f}^k$ 
for $e \in \Delta_0$, it is straightforward to verify that 
\begin{align*}
&\sum_{\substack{f \in \Delta_{m+1}(T)\\ f \supset g}} 
\sum_{e \in \Delta_1(f \cap g^*)} (\delta \phi)_{e}
\wedge L_g^* (\delta R^ku)_{e,f}\\
& = -
\sum_{\substack{f \in \Delta_{m+1}(T)\\ f \supset g}} 
\sum_{e \in \Delta_1(f \cap g^*)} \sum_{i \in I(e)} (-1)^{\sigma_e(x_i)}\lambda_{e(\hat x_i)} 
\sum_{p \in I(e)} (-1)^{\sigma_e(x_p)}
\wedge L_g^* A_{f(\hat x_p)}^k \\
&=\sum_{\substack{f \in \Delta_{m}(T)\\ f \supset g}} 
\sum_{ \substack{p \in I(f^*)\\i \in I(f\cap g^*)}}
(\lambda_p - \lambda_i)\wedge L_g^*A_f^k u
=\sum_{\substack{f \in \Delta_{m}(T)\\ f \supset g}} 
\sum_{ \substack{p \in I(g^*)\\i \in I(f\cap g^*)}}
(\lambda_p - \lambda_i)\wedge L_g^*A_f^k u,
\end{align*}
while
\[
-(n-s) \sum_{(e,f) \in \Delta(m,0)}
 \phi_e \wedge L_{g}^* R_{e,f}^ku  
= \sum_{\substack{f \in \Delta_{m}(T)\\ f \supset g}} 
 \sum_{ \substack{p \in I(g^*)\\i \in I(f\cap g^*)}}
\lambda_i \wedge L_g^*A_f^k u.
\]
As a consequence, the left hand side of \eqref{init-id} is 
\[
 \sum_{\substack{f \in \Delta_{m}(T)\\ f \supset g}} 
\sum_{ \substack{p \in I(g^*)\\i \in I(f\cap g^*)}}
\lambda_p \wedge L_g^*A_f^k u = (m-s) \rho_g 
\sum_{\substack{f \in \Delta_{m}(T)\\ f \supset g}} L_g^*A_f^k u,
\]
and this completes the proof.
\end{proof}
Note that if $k=0$, then from \eqref{Rkd-delta}, $R_{e,f}^1 du =
- (\delta R^0)_{e,f}$. As a consequence, formula \eqref{Cmgk-ident} follows from 
the result of Lemma~\ref{lem:Qrewrite-0} in the case $k=0$.

To prove Lemma~\ref{lem:Qrewrite}, we will also need the following identity.
\begin{lem}\label{lem:im-0}
The identity 
 \begin{multline}\label{im-0}
 \sum_{(e,f) \in \Delta(m,j)}
 d \Big(\frac{(\delta \phi)_{e}}{\rho_g^j}\Big) 
\wedge L_g^* R_{e,f}^k u + \frac{j}{\rho_g^{j+1}} \sum_{(e,f) \in \Delta(m,j+1)}
 (\delta \phi)_e
 \wedge L_g^* (\delta R^k u)_{e,f} 
\\
= \frac{j}{\rho_g^{j+1}} (n-s)  \sum_{(e,f) \in \Delta(m,j)} \phi_{e}  
\wedge L_g^* R_{e,f}^k u
\end{multline}
holds for any $0 \le j \le m$.
\end{lem}

The proof of this identity is technical, so we delay the proof until
we have used it to prove Lemma~\ref{lem:Qrewrite}.

\begin{proof} (of Lemma~\ref{lem:Qrewrite}) 
We introduce the operators 
  \[
  Q_{m,\ell}^k u  = \frac{1}{n-s}
\sum_{j = 1}^{\ell} (-1)^{j }(j-1)! \sum_{(e,f) \in \Delta(m,j)} 
(\delta \phi)_{e} \wedge L_g^* b^{-j} R_{e,f}^k u,
\]
such that $Q_{m,k}^k = Q_m^k$, and $Q_{m,0}^k = 0$.
We will now use induction with respect to $\ell$ to show that 
\begin{multline}
\label{Cmgk-ident-reduced}
C_{m,\ell}^k(g,T)u - \frac{n-m}{n-s} 
\sum_{\substack{f \in \Delta_m(T) \\ f \supset g}} L_g^* A_{f}^k u
= d Q_{m,\ell}^k u + Q_{m,\ell}^{k+1} du 
\\
+  \frac{(-1)^{\ell} \ell!}{n-s} \sum_{(e,f) \in \Delta(m,\ell +1)}
\frac{(\delta \phi)_{e}}{\rho_g^{\ell +1}}
\wedge L_g^* (\delta R^ku)_{e,f}, \quad \ell = 0, 1, \ldots k.
\end{multline}
For $\ell = 0$, this is exactly the identity given in
Lemma~\ref{lem:Qrewrite-0}.  On the other hand, for $\ell = k$, we have
that
\[
\Q_{m,k}^{k+1} du +   \frac{(-1)^k k!}{n-s} \sum_{(e,f) \in \Delta(m,k +1)}
\frac{(\delta \phi)_{e}}{\rho_g^{k +1}}
\wedge L_g^* (\delta R^ku)_{e,f} = Q_m^{k+1} du,
\]
where we have used the facts that $\rho_g = L_g^* b$ and $(\delta
R^ku)_{e,f} = - R_{e,f}^{k+1} du$ for $e \in \Delta_{k+1}(f)$, cf.
\eqref{Rkd-delta}. So the desired identity, \eqref{Cmgk-ident} ,
follows from \eqref{Cmgk-ident-reduced} with $\ell = k$.

If we assume that \eqref{Cmgk-ident-reduced} holds for $\ell -1$, then
\begin{align*}
C_{m,\ell}^k&(g,T) u - \frac{n-m}{n-s} 
\sum_{\substack{f \in \Delta_m(T) \\ f \supset g}} L_g^* A_{f}^k u
= (-1)^{\ell}\ell! \sum_{(e,f) \in \Delta(m,\ell )}
\frac{\phi_{e}}{\rho_g^{\ell+1 }}
\wedge L_g^* R_{e,f}^ku\\
&+ dQ_{m,\ell -1}^ku + Q_{m,\ell -1}^{k+1} du 
-   \frac{(-1)^{\ell} (\ell-1)!}{n-s} \sum_{(e,f) \in \Delta(m,\ell )}
\frac{(\delta \phi)_{e}}{\rho_g^{\ell }}
\wedge L_g^* (\delta R^ku)_{e,f}\\
&= dQ_{m,\ell -1}^ku + Q_{m,\ell }^{k+1} du \\
&+ \sum_{(e,f) \in \Delta(m,\ell )} \Big[(-1)^{\ell}\ell! 
\frac{\phi_{e}}{\rho_g^{\ell +1}}
\wedge L_g^* R_{e,f}^ku  
- \frac{ (\ell-1)!}{n-s} \frac{(\delta \phi)_{e}}{\rho_g^{\ell }}
\wedge dL_g^* R_{e,f}^ku \Big],
\end{align*}
where we have used \eqref{Rkd-delta} for the last equality. However, by 
\eqref{im-0}, the last sum can rewritten as 
\[
\frac{ (-1)^{\ell} (\ell-1)!}{n-s} \Big[
\sum_{(e,f) \in \Delta(m,\ell )} 
 d \Big(\frac{(\delta \phi)_{e}}{\rho_g^{\ell }}
\wedge L_g^* R_{e,f}^ku \Big) + \ell \hskip -7pt 
\sum_{(e,f) \in \Delta(m,\ell +1)}
\frac{(\delta \phi)_{e}}{\rho_g^{\ell +1}}
\wedge L_g^* (\delta R^ku)_{e,f} \Big],
\]
and hence we obtain the identity \eqref{Cmgk-ident-reduced} at level
$\ell$. This completes the induction argument, and hence the proof of
Lemma~\ref{lem:Qrewrite}.

\end{proof}

To complete the discussion of this section, leading to
Proposition~\ref{prop:preservation}, we need to establish the identity
\eqref{im-0}.
\begin{proof} (of Lemma~\ref{lem:im-0}) We observe that if $e \in
  \Delta_j(f \cap g^*)$, it follows from \eqref{d-phi} and the
  identity $\rho_g = \sum_{p \in I(g^*)} \lambda_p$ that
\begin{multline}\label{im-1}
d \Big(\frac{(\delta \phi)_{e}}{\rho_g^j}\Big) 
= \frac{j}{\rho_g^{j+1}} \sum_{i \in I(e)} (-1)^{\sigma_e(x_i)}
\sum_{p \in I(g^*)} \phi_{[x_p,e(\hat x_i)]}
\\
= \frac{j}{\rho_g^{j+1}} \Big[ (j+1)\phi_e 
+  \sum_{i \in I(e)} (-1)^{\sigma_e(x_i)}
\sum_{p \in I(g^*\setminus e)} \phi_{[x_p,e(\hat x_i)]} \Big].
\end{multline}
To proceed, we will treat the sum with respect to $p$ above in the two
cases $p \in I ((g^* \setminus e) \cap f)$ and $p \in I ((g^*
\setminus e) \cap f^* = I(f^*)$ separately.  In the first case, for
any fixed $f \in \Delta_{m+j}(T)$, consider
\begin{multline*}
\sum_{e \in \Delta_j(f \cap g^*)} \sum_{i \in I(e)} (-1)^{\sigma_e(x_i)}
\sum_{p \in I((g^* \setminus e)\cap f )} \phi_{[x_p,e(\hat x_i)]}\wedge L_g^* R_{e,f}^k u 
\\
= \sum_{e \in \Delta_{j+1}(f \cap g^*)} \sum_{p \in I(e)}\sum_{i \in I(e(\hat x_p))}
 (-1)^{\sigma_{e(\hat x_p)}(x_i) + \sigma_{e(\hat x_i)}(x_p)} \phi_{e(\hat x_i)} 
\wedge L_g^* R_{e(\hat x_p),f}^k u,
\end{multline*}
where the identity is obtained by introducing $e^\prime \in
\Delta_{j+1}$ as the ordered version of the simplex $[x_p,e]$,
i.e., $(-1)^{\sigma_{e^\prime}(x_p)}e^\prime = [x_p,e]$, and then
dropping primes.  However, it is easy to show that
\begin{equation}\label{sum-sigma}
\sigma_{e(\hat x_p)}(x_i) + \sigma_{e(\hat x_i)}(x_p)
= \sigma_{e}(x_i) + \sigma_{e}(x_p) -1.
\end{equation}
As a consequence, we can express the sum above as 
\begin{align*}
&\sum_{e \in \Delta_j(f \cap g^*)} \sum_{i \in I(e)} (-1)^{\sigma_e(x_i)}
\sum_{p \in I((g^* \setminus e)\cap f )} \phi_{[x_p,e(\hat x_i)]}\wedge L_g^* R_{e,f}^k u \\
&= - \sum_{e \in \Delta_{j+1}(f \cap g^*)} \sum_{p \in I(e)}
\sum_{i \in I(e)}
 (-1)^{\sigma_{e}(x_i) + \sigma_{e}(x_p)} \phi_{e(\hat x_i)} 
\wedge L_g^* R_{e(\hat x_p),f}^k u\\
 &\qquad + \sum_{e \in \Delta_{j+1}(f \cap g^*)} \sum_{p \in I(e)}
 \phi_{e(\hat x_p)} \wedge L_g^* R_{e(\hat x_p),f}^k u \\
 &=  - \sum_{e \in \Delta_{j+1}(f \cap g^*)} (\delta \phi)_e 
\wedge L_g^* (\delta^+ R^k u)_{e,f}
+ ( m-s-1 ) \sum_{e \in \Delta_{j}(f \cap g^*)} \phi_e \wedge L_g^* R_{e,f}^k u,
 \end{align*}
 where we have used the fact that for $f \in \Delta_{m+j}(T)$ and $g
\in \Delta_s(f)$, $|f \cap g^*| = m+j - s$. Choosing an 
$e \in \Delta_j(f \cap g^*)$ leaves $m+j -s - j-1 = m-s-1$ vertices
that can be deleted from an  $e^\prime \in \Delta_{j+1}(f \cap g^*)$
to produce that same $e$.
However, the first term on the right hand side vanishes since
$(\delta^+ R u)_{e,f} = 0$ by
Proposition~\ref{prop:R-relation}. Therefore, we can conclude that
 \begin{multline}\label{im-2}
 \sum_{e \in \Delta_j(f \cap g^*)} \sum_{i \in I(e)} (-1)^{\sigma_e(x_i)}
\sum_{p \in I((g^* \setminus e)\cap f )} 
\phi_{[x_p,e(\hat x_i)]}\wedge L_g^* R_{e,f}^k u
\\
= ( m-s-1 ) \sum_{e \in \Delta_{j}(f \cap g^*)} \phi_e \wedge L_g^* R_{e,f}^k u.
\end{multline}
 In an analogous manner, and by using the identity \eqref{sum-sigma} as
above, we obtain
 \begin{align*}
\sum_{(e,f) \in \Delta(m,j)}& \sum_{i \in I(e)} (-1)^{\sigma_e(x_i)}
\sum_{p \in I(f^* )} \phi_{[x_p,e(\hat x_i)]}\wedge L_g^* R_{e,f}^k u \\
 &= -\sum_{(e,f) \in \Delta(m,j+1)} \sum_{p \in I(e)}\sum_{i \in I(e)}
 (-1)^{\sigma_{e}(x_i) + \sigma_{e}(x_p)} 
 \phi_{e(\hat x_i)} 
 \wedge L_g^* R_{e(\hat x_p),f(\hat x_p)}^k u\\
 &+ \sum_{(e,f) \in \Delta(m,j+1)} \sum_{p \in I(e)}   \phi_{e(\hat x_p)} 
 \wedge L_g^* R_{e(\hat x_p),f(\hat x_p)}^k u,
 \end{align*}
 where as above 
 we have introduced $(-1)^{\sigma_{e^\prime}(x_p)} e^\prime = [x_p,e]$,
 and the corresponding extension of $f$ to $f^\prime \in
 \Delta_{m+j+1}$ by including $x_p$.  However, the final right hand
 side above can rewritten as
\[
 -\sum_{(e,f) \in \Delta(m,j+1)}
 (\delta \phi)_e
 \wedge L_g^* (\delta R^k u)_{e,f}
 + (n- m -j)\sum_{(e,f) \in \Delta(m,j)}
  \phi_{e} 
 \wedge L_g^* R_{e,f}^k u.
\]
 In this case,  for each $f \in \Delta_{m+j}(T)$, there are 
$n-m-j$ vertices that can be deleted from $f^\prime \in \Delta_{m+j+1}(T)$ 
to produce the same $f$. Deleting this same vertex from $e^{\prime}
\in \Delta_{j+1}(f^\prime \cap g^*)$ produces the above result.

By combining this result with \eqref{im-1} and \eqref{im-2}, we obtain
\begin{multline*}
  \sum_{(e,f) \in \Delta(m,j)} d \Big(\frac{(\delta \phi)_{e}}{\rho_g^j}\Big) 
\wedge L_g^* R_{e,f}^k u 
\\
=\frac{j}{\rho_g^{j+1}} \Big[ (n-s) \sum_{(e,f) \in \Delta(m,j)}
 \phi_e\wedge L_g^* R_{e,f}^k u 
-\sum_{(e,f) \in \Delta(m,j+1)} 
 (\delta \phi)_e
 \wedge L_g^* (\delta R^k u)_{e,f} \Big],
\end{multline*}
which is exactly the desired identity.
\end{proof}

\begin{remark}
  By a careful inspection of the  proofs of
    Lemmas~\ref{lem:support-Ck} and \ref{lem:Qrewrite}, we will
  discover that all properties of the operators $A_f^k$ and
  $R_{e,f}^k$ are used, except for the trace preserving property given
  by statement iii) of Lemma~\ref{lem:A-prop}, i.e., that $\tr_f A_f^k u =
  \tr_f u$. In fact, this property is only used to establish the
  identity \eqref{decomp-id}.  In future work, we will consider the
  possibility of constructing approximations of a form $u$ by using
  its decomposition by the bubble transform,
  cf. \eqref{decomp-id}. One direct way to construct such an
  approximation is to approximate the operator $C_m^k$, studied above,
  by an operator of the form
\begin{multline*}
\tilde C_m^k u = \sum_{f \in \Delta_m(\T)} 
\sum_{g \in \bar \Delta(f)}(-1)^{|f| -|g|}L_g^* \tilde A_f^k u 
\\
+  \sum_{\substack{f \in \Delta_{m+j}(\T)\\ 0 \le j \le k}} j!  
\sum_{g \in \bar \Delta(f)}(-1)^{|f| -|g|}\sum_{e \in \Delta_j(f \cap g^*)}
 \frac{\phi_e}{\rho_g} \wedge L_{g}^* b^{-j}\tilde R_{e,f}^ku,
\end{multline*}
i.e., we have replaced the operators $A_f^k$ and $R_{e,f}^k$ by
corresponding approximations $\tilde A_f^k$ and $\tilde R_{e,f}^k$.
By the observation above, we can conclude that if these operators
satisfy the two relations \eqref{Rkd-delta} and \eqref{Rk-delta+},
then the operator $\tilde C_m^k$ commutes with the exterior
derivative. Furthermore, piecewise polynomial properties of the
functions $\tilde C_m^ku$ and the support properties of
the corresponding operators $\tilde C_{m,f}^ku$ can be derived from similar
properties of the operators $\tilde A_f^k$ and $\tilde R_{e,f}^k$
\end{remark}

\section{Bounding the operator norms}
\label{sec:bounds}
The constructions above are derived under the assumptions given in Section~\ref{sec:assumptions}.
However, to give rigorous proofs of the estimates stated below, we will in this final section make the additional 
assumption that the manifold $x_i^*$ is connected for each $x_i \in \Delta_0(\T)$.
We note this will be the case
if $\Omega$ is a Lipschitz domain.

The various constants that appear in the bounds below only depend on
  the mesh $\T$ through the shape--regularity constant $c_{\T}$,
  defined by \eqref{shape-constant}. The consequence of this is that
  if we consider a family of meshes, $\{\T^h \}$, parametrized by a
  real parameter $h \in (0,1]$, typically obtained by mesh
  refinements, the bounds will be uniform with respect to $h$ as long
  as we restrict to a family with a uniform bound on the constants
  $\{c_{\T^h} \}$.  In the bounds we derive below, the various
  constants that appear will depend on the space dimension $n$ and the
  domain $\Omega$, in addition to the dependence explicitly stated.
  Throughout this section we will assume that the operators under
  investigation are applied to piecewise smooth differential forms.
  However, since the space $\Lambda^k(\T)$ is dense in
  $L^2\Lambda^k(\Omega)$, it a consequence of the bound obtained in
  Theorem~\ref {thm:L2bound} that all the operators $B_{m,f}^k$ and
  $B_m^k$ can be extended to bounded operators mapping
  $L^2\Lambda^k(\Omega)$ to itself.

\subsection{The main bounds}
\label{sec:main-bounds}
If $u$ is a $k$ form, we let $|u_x|$ be defined by 
\begin{equation*}
|u_x| = \sup u_x(t_1, \ldots, t_k),
\end{equation*}
where the $\sup$ is taken over all collections of unit tangent
vectors. As a consequence,
\[
\| u \|_{L^2(\Omega)} = \Big( \int_{\Omega} |u_x|^2 \, dx \Big)^{1/2}.
\]
Our estimates will use the domains $\Omega_{e,f}$, defined in
Section~\ref{weight-func} above, consisting of finite unions of $n$
simplexes in $\Delta_n(\T)$ and the extended macroelements
$\Omega_f^E$, consisting of the union of the macroelements associated
to the vertices of $f$.  The bounds for the operators $B_{m,f}^k$ and
$B_m^k$ will be obtained from the following bound for the cut--off
operator $C_{m,f}^k$.

\begin{lem}
\label{lem:C-bound}
There exists a constant $c$, depending on the mesh $\T$ only through
the shape-regularity constant $c_{\T}$, such that for $f \in
\Delta_{m+j}(\T)$ and $e \in \Delta_j(f)$, we have
\begin{equation}
\label{Cmfkbound}
\|C_{m,f}^k u\|_{L^2(\Omega_f)} \le c \|u\|_{L^2(\Omega_{f}^{E})},
\end{equation}
where $0 \le m \le n$ and $0 \le j \le k$.
\end{lem}
In addition to this result, the proof of the desired bounds will
depend on bounds for the overlap of the sets $\{\Omega_f \}$ and
$\{\Omega_{f}^{E} \}$. In the present setting, the overlap of a set
of subdomains can be defined as the smallest upper bound for the
number of domains which will contain any fixed element $T \in
\Delta_n(\T)$.  Alternatively, the overlap of the set is the
$L^\infty$ norm of the sum of the characteristic functions of the set.
The overlap of the set of macroelements, $\{ \Omega_f \}_{f \in
  \Delta_m(\T)}$, will only depend on $m$ and the space dimension $n$,
while the overlap for the sets $\{\Omega_{f}^{E} \}$ will depend on
the mesh $\T$, as established in the following result.

\begin{lem}\label{lem:overlap}
  The overlap of the domains $\{\Omega_f^E \}_{f \in \Delta(\T)}$
  can be bounded by a constant which depends on the mesh $\T$ only
  through the shape regularity constant $c_{\T}$.
\end{lem}

We will defer the proof of the two lemmas above until after the proof of the
main results given in this section. 
\begin{thm}
\label{thm:L2bound}
There exists a constant $c$, depending on the shape-regularity
constant $c_{\T}$, such that for $0 \le m \le n$, we have
\[
 \|B_{m}^k u\|_{L^2(\Omega)}, 
 \Big(\sum_{j=0}^k \sum_{f \in \Delta_{m+j}(\T)} 
\|B_{m,f}^k u\|_{L^2(\Omega)}^2 \Big)^{1/2} 
 \le c \|u\|_{L^2(\Omega)}.
 \]
\end{thm}
\begin{proof}
We recall that the the operator $C_m^k$ is defined by
\[
C_m^k u = \sum_{f \in \Delta[m,k]} C_{m,f}^ku,
\]
where, to simplify notation, we have introduced the set $\Delta[m,k] =
\{ f \in \Delta_{m+j}(\T) \, : \, 0 \le j \le k \,\}$.  We will first
show that
\begin{equation}\label{Cm-bound}
\| C_m^k u \|_{L^2(\Omega)} \le c_1 \| u \|_{L^2(\Omega)},
\end{equation}
where the constant $c_1$ depends on $c_{\T}$. To see this, let
$\kappa_f$ be the characteristic function of the set $\Omega_f$. Since
the functions $C_{m,f}^k u$ have support in $\Omega_f$,
cf. Lemma~\ref{lem:support-Ck}, we have by repeated use of the Cauchy-Schwarz
inequality, that
\begin{align*}
\| C_m^k u &\|_{L^2(\Omega)}^2 = \sum_{f,g \in \Delta[m,k]} 
\int_{\Omega} \kappa_f \kappa_g |C_{m,f}^ku| \, |C_{m,g}^ku|  \, dx\\
&\le \sum_{f,g \in \Delta[m,k]} 
(\int_{\Omega} \kappa_f \kappa_g |C_{m,f}^ku|^2  \, dx)^{1/2}
(\int_{\Omega} \kappa_f \kappa_g |C_{m,g}^ku|^2  \, dx)^{1/2} \\
&\le \Big(\sum_{f,g \in \Delta[m,k]} 
\int_{\Omega} \kappa_f \kappa_g |C_{m,f}^ku|^2  \, dx \Big)^{1/2}
 \Big(\sum_{f,g \in \Delta[m,k]} 
\int_{\Omega} \kappa_f \kappa_g |C_{m,g}^ku|^2  \, dx \Big)^{1/2}\\
&\le \alpha_0 \sum_{f  \in \Delta[m,k]} \| C_{m,f}^k u \|_{L^2(\Omega_f)}^2,
\end{align*}
where $\alpha_0$ is the overlap of set $\{ \Omega_f \}_{f \in \Delta[m,k]}$.
However, by the bound \eqref{Cmfkbound}, we have 
\begin{equation}\label{square-Cm}
\sum_{f  \in \Delta[m,k]} \| C_{m,f}^k u \|_{L^2(\Omega_f)}^2
\le c^2 \sum_{f  \in \Delta[m,k]} \| u \|_{L^2(\Omega_f^E)}^2
\le \alpha_1 c^2 \| u \|_{L^2(\Omega)}^2,
\end{equation}
where $\alpha_1$ is the overlap of the set $\{ \Omega_f^E \}_{f \in
  \Delta[m,k]}$, cf. Lemma~\ref{lem:overlap}.  Hence, we have verified
the bound \eqref{Cm-bound}.  The desired bound for the functions
$B_m^k u$, $0 \le m \le n$, now follows from this bound, the iteration
\eqref{Bmkdef}, and a simple induction argument with respect to
$m$. Finally, the $L^2$ bound the functions $B_{m,f}^k u$ follows from
the bound on the functions $B_m^k u$, \eqref{def-Bm}, and
\eqref{square-Cm}.
\end{proof}

Combining Theorem~\ref{thm:L2bound} with the fact that the operators
$B_m^k$ commute with the exterior derivative, cf. Theorem
\ref{thm:trBmk}, we also obtain a bound on the operators $B_m^k$ in
the norm $\|\cdot \|_{H \Lambda(\Omega)}$, where
\begin{equation*}
\|u\|_{H \Lambda(\Omega)} = (\|u\|_{L^2(\Omega)}^2 + (\|d u\|_{L^2(\Omega)}^2)^{1/2}.
\end{equation*}
\begin{thm}
\label{thm:Hbound}
There exists a constant
$c$, depending on the shape regularity constant $c_{\T}$, such that
\begin{equation*}
 \|B_{m}^k u\|_{H \Lambda(\Omega)} \le c \|u\|_{H \Lambda(\Omega)}, \quad 0 \le m \le n.
\end{equation*}
\end{thm}
\begin{proof}
  Since $d B_{m}^k u = B_{m}^{k+1} du$, this is a direct consequence of
  the $L^2$ bounds given in Theorem~\ref{thm:L2bound}.
\end{proof}

\subsection{Deriving the bounds}
\label{sec:der-bounds}
To complete the proofs of the main results above, we need to prove 
Lemmas~\ref{lem:C-bound} and \ref{lem:overlap}.
We will first present the proof of Lemma~\ref{lem:overlap}.

\begin{proof} (of Lemma~\ref{lem:overlap})
For each $x \in \Delta_0(\T)$, we let $N_x$
be  the number of $n$ simplices containing the vertex $x$.
We will show that the number $N_{x}$ can be bounded from above
by a constant which only depends on $\T$ though the shape-regularity
constant $c_{\T}$. In fact, for any vertex $x_0$ we have 
\[
N_{x_0}= \sum_{T \in \Delta_n(\T_{x_0})}  
\le  \sum_{T \in \Delta_n(\T_{x_0})}  \frac{ |T|}{|\Ball_T|}
= \sum_{T \in \Delta_n(\T_{x_0})}   \frac{ h_T^n}{|\Ball_T|} h_T^{-n}|T|,
\]
where $h_T$ is the diameter of the $n$ simplex $T$ and $\Ball_T$ is
the largest ball contained in $T$.  Next we use the fact that
$|\Ball_T | = \beta_n (\diam(\Ball_T)/2)^n$, where $\beta_n$ is the
volume of the unit ball in $\R^n$ to obtain
\[
N_{x_0} \le \beta_n^{-1} 2^n \sum_{T \in \Delta_n(\T_{x_0})} 
\frac{ h_T^n}{\diam(\Ball_T)^n} h_T^{-n}|T|
\le \beta_n^{-1} (2 c_{\T})^n \sum_{T \in \Delta_n(\T_{x_0})} h_T^{-n}|T|,
\]
where we have used the definition of $c_{\T}$ for the last
inequality. However, by substituting $\theta(x) = (x - x_0)/h_T$ for
$x \in T$, we obtain
\[
 \sum_{T \in \Delta_n(\T_{x_0})} h_T^{-n}|T|
 \le \int_{|\theta| \le 1} \, d\theta = \beta_n.
 \]
Hence, we can conclude that 
\begin{equation}\label{N-bound}
N_{x_0} \le (2 c_{\T})^n.
\end{equation}
Note that it follows from \eqref{Omega-e-f} that if $g \subset f$ then
$\Omega_g^E \subset \Omega_f^E$.  Therefore, to derive an upper
bound for the overlap of the set $\{ \Omega_f^E \}$, it is enough to
consider the sets $\{\Omega_f^E \}_{f \in \Delta_n(\T)}$. However,
if $T$ is any fixed $n$ simplex, then $T$ is a subset of
$\Omega_f^E$ if and only if $T \cap f$ contains at least one vertex.
As a consequence, $T$ belongs to at most $(n+1) \max_{x \in
  \Delta_0(T)} N_x$ domains of the set $\{\Omega_f^E\}_{f \in
  \Delta_n(\T)}$, and therefore the desired bound follows from
\eqref{N-bound}.
\end{proof}
It remains to prove Lemma~\ref{lem:C-bound}.  To do so, will require
several preliminary results. We begin with a discussion of some
further consequences of shape-regularity.  By using the fact that the
volume of $\Ball_T$, $|\Ball_T|$, is less than $|T|$, we obtain the
estimate
\[
h_T^n \le \beta_n^{-1}(2c_{\T})^n |\Ball_T| \le \beta_n^{-1}(2c_{\T})^n |T|,
\]
where the constant $\beta_n$ is the same constant as in the proof above.
In fact, if $f \in \Delta_m(T)$, then we can utilize the natural
projection from $T$ to $f$, given by
\[
\sum_{i \in I(T)} \lambda_i(x)x_i \mapsto  
\sum_{i \in I(f)} \lambda_i(x)x_i/[\sum_{i \in I(f)} \lambda_i(x)],
\]
to obtain 
the more general estimate 
\begin{equation}\label{vol-f-bound}
h_T^m  \le \beta_m^{-1} (2c_{\T})^m |f|,
\end{equation}
where $|f|$ is the $m$ dimensional volume of $f$.
A further consequence of shape-regularity is local quasi-uniformity of
the mesh. In particular, we have the following result for the
macroelements $\Omega_{e,f}$.
\begin{lem}\label{lem:max-min-vertex}
There is a
  constant $c$, depending on $\T$ only through the shape-regularity
  constant $c_{\T}$, such that
\begin{equation}\label{max-min}
\max_{T \in \Delta_n(\T_{e,f})} h_T \le
c \min_{T \in \Delta_n(\T_{e,f})} h_T, 
\quad f \in \Delta(\T), \, e \in \Delta(f).
\end{equation}
\end{lem}
\begin{proof}
We first prove that
\[
\max_{T \in \Delta_n(\T_{x_i})} h_T \le
c \min_{T \in \Delta_n(\T_{x_i})} h_T,  \quad x_i \in \Delta_0(\T).
\]
To do so, let $T_-$ and $T_+$ be two $n$-simplices in $\Omega_{x_i}$, and
assume that there is a finite sequence of $n$ simplexes $\{T_j
\}_{j=0}^s$ in $\Omega_{x_i}$ such that $T_- = T_0,\, T_s = T_+$ and
$T_j \cap T_{j+1}$ contains at least one element $e \in \Delta_1(\T)$
containing $x_i$. By repeated use of the inequality
\eqref{vol-f-bound} with $m=1$, we then obtain
\[
\max(h_{T_-}, h_{T_+}) \le (2 c_{\T})^s \min(h_{T_-}, h_{T_+}). 
\]
However, since we have assumed that $x_i^*$ is
connected, any two $n$ simplexes $T_-$ and $T_+$ in $\Omega_{x_i}$ can
be connected by a sequence of the form above. Furthermore, as a
consequence of Lemma~\ref{lem:overlap}, the number $s$ can be bounded
by a constant which only depends on $\T$ through the shape-regularity
constant.

Since $\Omega_{e,f} \subset \Omega_{f,f} = \Omega_f^E$,
to prove \eqref{max-min},
it is enough to prove the result for $\Omega_{f,f}$. Now
for each $f \in \Delta(\T)$, we have
\[
\bigcup_{i \in I(f)} \Omega_{x_i} = \Omega_f^E, \quad \text{and } 
\bigcap_{i \in I(f)} \Omega_{x_i} = \Omega_f \neq \emptyset.
\]
Suppose $\max_{T \in \Delta_n(\T_{f,f})} h_T$ occurs for $T \in \T_{x_i}$
and $\min_{T \in \Delta_n(\T_{f,f})} h_T$ occurs for $T \in \T_{x_j}$.
Then by the result above for $\Omega_{x_i}$,
\begin{multline*}
\max_{T \in \Delta_n(\T_{f,f})} h_T = \max_{T \in \Delta_n(\T_{x_i})} h_T
\le c \min_{T \in \Delta_n(\T_{x_i})} h_T \le c \min_{T \in \Delta_n(\T_{f})} h_T
\\
\le c \max_{T \in \Delta_n(\T_{f})} h_T \le  c \max_{T \in \Delta_n(\T_{x_j})} h_T
 \le c^2 \min_{T \in \Delta_n(\T_{x_j})} h_T = c^2 \min_{T \in \Delta_n(\T_{f,f})} h_T.
\end{multline*}
This completes the proof of the lemma.
\end{proof}

Next, recall that the operator $L: \Omega \to \S$ is defined by
\begin{equation*}
L x =\{\lambda(x_i)\}_{i \in \I}.
\end{equation*}
If we apply the map $L$ to an $n$ simplex $T \in \Delta_n(\T)$, we
obtain a corresponding $n$ simplex $L(T) \subset \S$.  More precisely,
if $T = [x_{j_0}, \ldots ,x_{j_n}]$ then $L(T) = [e_{j_0}, \ldots
,e_{j_n}]$, where $e_i = Lx_i$ corresponds to unit vectors in
$R^{N+1}$, where $N+1$ is the number of elements in
$\Delta_0(\T)$.  The operator $L$ restricted to $T$, $L_T$, has an
inverse $F = F_T$. More precisely,
\[
L_Tx = \sum_{i \in I(T)}\lambda_i(x) e_i, \quad \text{and } F_T
\lambda = \sum_{i \in I(T)} \lambda_i x_i.
\]
Furthermore, $DL_T = D_x L_T $ satisfies $DL_T(x_i - x_j) = (e_i
-e_j)$ for $i,j \in I(T)$.  The shape regularity constant $c_{\T}$ can
be used to bound $DL_T$. More precisely, we can easily derive the bound
\begin{equation}\label{DL-bound}
\| DL_T \| \le c_{\T} h_{L(T)} h_T^{-1} \le 2 c_{\T} h_T^{-1},
\end{equation}
where $\| \cdot \|$ is the operator norm corresponding to the
Euclidean vector norm, and where $h_T$ and $h_{L(T)}$ denote the
diameter of $T$ and $L(T)$, respectively. This bound can, for
example, be found in \cite[Theorem 3.1.3]{Ciarlet}. 
For each $f \in \Delta(\T)$ and $e \in \bar \Delta(f)$, we define 
$\S_{e,f} \subset \S$ by
\[
\S_{e,f} = \bigcup_{\substack{T \in \Delta_n(\T)\\ T \subset \Omega_{e,f}}} L(T).
\]
Hence, $\S_{e,f}$ is an $n$ dimensional manifold such that all $n$
simplexes of $\S_{e,f}$ contain $\S_{f \cap e^*}$ as a subcomplex.
Furthermore, restricted to $\S_{e,f}$, the map $L$ can be inverted,
with an inverse $F_{e,f} : \S_{e,f} \to \Omega_{e,f}$ given by
\[
F_{e,f} \lambda = F_T \lambda, \quad \lambda \in L(T).
\]
In order to establish Lemma~\ref{lem:C-bound}, we will need a bound for the functions $z_{e,f}$, 
constructed in
Section~\ref{weight-func} to define the order reduction operators $R_{e,f}^k$. 

\begin{lem}
  \label{lem:zbound} There exists a constant $c$, depending on the
  mesh $\T$ only through the shape regularity constant $c_{\T}$, such
  that
\begin{equation*}
\|z_{e,f}\|_{L^{\infty}(\Omega_{e,f})} \le c h_{e,f}^{j-n}, \qquad e \in \Delta_j(f),
\end{equation*}
where $h_{e,f} = \max_{T \subset \Delta_n(\T_{e,f})} h_T$.
\end{lem}

\begin{proof} 
  Recall that the functions $z_{e,f}$ are defined by $z_{e,f} =
  (\delta^+w)_{e,f}$, where $w_{e,f} \in
  \0\P_1^-\Lambda^{n-j-1}(\T_{e,f})$ for $e \in \Delta_j(f)$, $j \ge
  0$.  The desired bound on the functions $z_{e,f}$ will be derived
  from a corresponding bound on the functions $w_{e,f}$, and to obtain
  this bound, we will use a scaling argument.  For each $e \in \bar
  \Delta(f)$, we define $\tilde w_{e,f} = F_{e,f}^*w_{e,f}$, such that
  $w_{e,f} = L^*\tilde w_{e,f}$.  From the process defining the
  functions $w_{e,f}$, we obtain that the functions $\tilde w_{e,f}$
  are uniquely specified by a corresponding process on $\S$.  In
  particular, the initial functions $\tilde w_{\emptyset,f}$ are
  piecewise constants with integral equal to minus one,
\[
d\tilde w_{e,f} =(-1)^{j}( (\delta - \delta^+)\tilde w)_{e,f},
\]
and condition \eqref{w-orth-cond} translates to the corresponding relation
\[
\int_{\S_{e,f}} \tilde w_{e,f} \wedge \star dq = 0, 
\quad q \in  \0\P_1^-\Lambda^{n-j-2}(\S_{e,f}).
\]
Since the simplex $\S$ is of unit size, and since the number of $n$
simplexes belonging to the manifolds $\S_{e,f}$ is bounded by the
shape regularity constant, we can conclude that
\begin{equation}\label{tilde-w-bound}
\| \tilde w_{e,f}  \|_{L^\infty(\S)} \le c, \quad f \in \Delta(\T), 
e \in \Delta(f),
\end{equation}
where the constant $c$ depends on  $c_{\T}$.
Finally, we use the fact that for $e \in \Delta_j(f)$, the 
$n-j$ form  $z_{e,f}$ satisfies the relation
\[
z_{e,f} = L^* (\delta^+\tilde w)_{e,f}.
\] 
By the definition of the pullback $L^*$, we then obtain from \eqref{DL-bound}
and \eqref{tilde-w-bound}  that 
\[
\|z_{e,f}\|_{L^{\infty}(\Omega_{e,f})} \le c [\min_{T \subset \Delta_n(\T_{e,f})} h_T]^{j-n}
\le c [\max_{T \subset \Delta_n(\T_{e,f})} h_T]^{j-n},
\]
where we have used the inequality \eqref{max-min} in the last step.
\end{proof}

To prove Lemma~\ref{lem:C-bound}, we first recall some notation and
formulas developed in \cite{bubble-I}. If $f \in \Delta_m(\T)$ 
and $0 \le m \le n-1$, then we can write $x \in \Omega_f$ in the form
\begin{equation*}
x = \sum_{i \in I(f)} \lambda_i(x) x_i + \rho_f(x) q_f(x), \qquad q_f(x) \in f^*,
\end{equation*}
where $f^*$ is a piecewise flat manifold of dimension $n-m-1$, see also 
Section~\ref{sec:prelims} above.
As it was done in \cite[Section 5]{bubble-I}, we can use the mapping
$ x \mapsto (L_f(x), q_f(x))$ 
to express integrals over $\Omega_f$ as integrals over $\S_f^c \times f^*$.
In particular, if $\Omega_f^{\prime} \subset \Omega_f$ is a union of $n$ 
simplexes belonging to $\Omega_f$, then we have 
\begin{equation}\label{covform}
\int_{\Omega_f^\prime} \phi(L_f(x), q_f(x)) \, dx
= \int_{\S_f^c} \int_{f^*\cap \Omega_f^\prime} \phi(\lambda,q) J(f,q) 
\, dq \, b(\lambda)^{n-m-1}
\, d \lambda,
\end{equation}
for any sufficiently regular and real-valued function $\phi$ defined on $\S_f^c
\times f^*$.  Here $dq $ means integration with respect to the
standard Lebesgue measure derived from the imbedding of the tangent
space of $f^*$ into $\R^{n-m-1}$. The determinant $J(f,q)$ is a real valued 
piecewise constant function with respect to $q$. 
If $f = [x_0, x_1, \ldots ,x_m]$, then 
\[
J(f,q) = \det( [x_0 - \hat q, x_1 -\hat q, \ldots, x_m -\hat q,
t_{m+1}, \ldots,  t_{n-1}]),
\]
where $\hat q = \hat q(q)$ is the barycenter of $f^* \cap T$ for $q
\in f^* \cap T$ and any $n$ simplex $T \subset \Omega_f$. Furthermore,
$t_{m+1}, \ldots, t_{n-1} \in \R^n$ is an orthonormal basis for the
tangent space of $f^* \cap T$.  It follows from \eqref{covform}, with
$\phi \equiv 1$, that if $T \in \Delta_n(\T_f)$,
that
 \[
 \frac{|T|}{| f^* \cap T|} = \Big(\int_{\S_f^c} b(\lambda)^{n-m-1}\, 
d \lambda\Big)   J(f,q)|_T.
 \]
 However, the estimate \eqref{vol-f-bound} implies that the fraction
 $|T|/|f^* \cap T|$ can be bounded, above and below, by $h_T^{m+1}$
 times constants which depend on $c_{\T}$.  As a consequence of the
 bound \eqref{max-min}, we can therefore conclude that there exist
 constants $c_1$ and $c_2$, depending on the shape--regularity
 constant $c_{\T}$, such that
\begin{equation}\label{det-ineq}
c_1 h_f^{m+1} \le J(f,q) \le c_2 h_f^{m+1}, \quad q \in f^*,
\end{equation}
where $h_f = \max_{T \in\Delta_n( \T_f)} h_T$. 

\begin{proof}(of Lemma~\ref{lem:C-bound})
Recall that the operator $C_{m,f}^k$ is defined by 
\[
C_{m,f}^ku =   \sum_{g \in \bar \Delta(f )}
(-1)^{|f| - |g|}\frac{\rho_f}{\rho_g}
\wedge L_g^* b^{-j} A_f^k u,
\]
if $f \in \Delta_m(\T)$, and by
\begin{equation*}
C_{m,f}^ku =   j! \sum_{e \in \Delta_j(f)} \sum_{g \in \bar \Delta(f \cap e^*)}
(-1)^{|f| - |g|}\frac{\phi_{e}}{\rho_g}
\wedge L_g^* b^{-j} R_{e,f}^k u,
\end{equation*}
if $f \in \Delta_{m+j}(\T)$, $1\le j\le k$.  If $m=n$, such that $f$
is an $n$ simplex, then $\tr_f C_{m,f}^k = \tr_f$ and the
conclusion of the lemma obviously holds. Therefore, we can assume that
$0 \le m \le n-1$ in the rest of the proof.

The function $C_{m,f}^ku$
has support on $\Omega_f$, and for $x \in \Omega_f$ and $g \in \bar
\Delta(f)$, we have $\rho_f/\rho_g \le 1$. Furthermore, it is a
consequence of \eqref{max-min} that
 \[
 |\phi_e/\rho_g| \le c h_f^{-j}, \quad e \in \Delta_j(f \cap g^*),
 \]
where the constant $c$ depends  on the shape-regularity constant.
  Therefore, since $\Omega_f \subset \Omega_{e,f}$, to prove an
 inequality of the form \eqref{Cmfkbound}
for the case $f \in \Delta_{m+j}(\T)$, it will be
sufficient to show that
 \begin{equation}
\label{keyterm-bound}
\|L_g^*[b^{-j} R_{e,f}^ku]\|_{L^2(\Omega_f)} \le c h_{e,f}^j \|u\|_{L^2(\Omega_{e,f})}, 
\qquad e \in \Delta_j(f), \quad g \in \bar \Delta(f \cap e^*),
\end{equation}
where $h_{e,f} = \max_{T \subset \T_{e,f}} h_T$.  Here we recall from
Section~\ref{sec:Refk} that the operator $R_{e,f}^k$ is defined by
\[
(R_{e,f}^k u)_{\lambda}
=  \int_{\Omega_{e,f}} (\Pi_j G^*u)_{\lambda} \wedge z_{e,f}.
\]
However, for any $e \in \Delta_0(f)$, $R_{e,f}^k u$ corresponds to the
operator $A_f^k u$, so the desired bound, \eqref{Cmfkbound}, for the
case $f \in \Delta_m(\T)$, will follow from \eqref{keyterm-bound} with
$j=0$.

To show the bound \eqref{keyterm-bound}, we assume that $f \in
\Delta_{m+j}(\T)$, $e \in \Delta_j(f)$ such that $f \cap e^* \in
\Delta_{m-1}(\T)$ and $g \in \Delta_s(f \cap e^*)$ for $0 \le s \le
m-1$. We also need to treat the case $g = \emptyset$, but this will be
done as a special case below.  We will use formula \eqref{covform}
with $f$ replaced by $g$.  In this case, it follows from
\eqref{det-ineq} that the determinant $J(g,q) = O(h^{s+1})$,
where here, and in the rest of this proof $h = h_{e,f}$.
Furthermore, $g^*$ is an $n-s-1$ dimensional manifold of size $h$, so
its volume, $|g^*| = O(h^{n-s-1})$. Therefore, since $\Omega_f \subset
\Omega_{g}$, and noting that $b^{-j} R_{e.f}^ku$ only depends on
$\lambda$, we have from \eqref{covform} that
 \begin{equation}\label{intermed}
 \| L_{g}^*[b^{-j} R_{e.f}^ku] \|_{L^2(\Omega_f)} 
\le c \Big[h^n \int_{\S_{g}^c} b(\lambda)^{n-s-1} 
 \Big(b(\lambda)^{-j} |(R_{e,f}^k u)_{\lambda}|\Big)^2 \, d\lambda \Big]^{1/2},
 \end{equation}
 where the constant $c$ only depends on $\T$ through the shape
 regularity constant $c_{\T}$.  By using the fact that
 \[
 D_{\lambda} G = \sum_{i \in I(g)} (x_i - y) d\lambda_i
 \]
 is uniformly bounded for $y \in \Omega_{e,f}$, and that $D_y G$ is
 $b(\lambda)$ times the identity, we obtain
 \[
 b(\lambda)^{-j}|(R_{e,f}^k u)_{\lambda}| \le c \int_{\Omega_{e,f}} |u_{G(y, \lambda)}| 
\, |(z_{e,f})_y| \, dy 
 \le c h^{j-n} \int_{\Omega_{e,f}} |u_{G(y, \lambda)}|  \, dy,
 \]
 where we have used the result of Lemma~\ref{lem:zbound}
 for the final inequality.
 Furthermore, since 
 $\Omega_{e,f} \subset \Omega_{f \cap e^*} \subset \Omega_g$,
we have from \eqref{covform}  and \eqref{det-ineq} that
 \[
 b(\lambda)^{-j}|(R_{e,f}^k u)_{\lambda}|
 \le c h^{j +s+1-n} \int_{\S_{g}^c} b(\mu)^{n-s-1} 
\int_{g^* \cap \Omega_{e,f}} | u_{G(G(q,\mu), \lambda)} | \, dq \, d\mu.
 \]
 By inserting this inequality into \eqref{intermed} and using
 Minkowski's integral inequality, we obtain
 \begin{align*}
  &\| L_{g}^*[b^{-j} R_{e.f}^ku] \|_{L^2(\Omega_f)} \\
  &\le c \Big[ h^n \int_{\S_{g}^c} b(\lambda)^{n-s-1}
  \Big(h^{j+s+1-n}  \int_{\S_{g}^c} b(\mu)^{n-s-1} \int_{g^* \cap \Omega_{e,f}} 
\hskip -6pt | u_{G(G(q,\mu), \lambda)} | \, dq \, d\mu\Big)^2
 \, d\lambda \Big]^{1/2}\\
\\
 &\le c h^{j+s+1-n/2} \int_{\S_{g}^c} b(\mu)^{n-s-1} 
\Big[ \int_{\S_{g}^c}b(\lambda)^{n-s-1}
 \Big( \int_{g^* \cap \Omega_{e,f}} \hskip -6pt
| u_{G(q, \lambda^\prime(\lambda,\mu))} | \, dq \Big)^2 
\, d\lambda \Big]^{1/2} \, d\mu.
 \end{align*}
 Here we have used the fact that 
 \[
 G(G(q,\mu),\lambda) = G(q, \lambda^\prime), \quad \text{where }
 \lambda^\prime(\lambda,\mu) = \lambda + b(\lambda) \mu.
 \]
 Next we introduce the change of variables $\lambda \to
 \lambda^\prime$, where $\det (d\lambda^\prime/d\lambda) = b(\mu)$ and
 $b(\lambda^\prime) = b(\lambda) b(\mu)$.  We obtain
 \begin{align*}
  &\| L_{g}^*[b^{-j} R_{e.f}^ku] \|_{L^2(\Omega_f)} \\
 &\le c h^{j+s+1-n/2}  \int_{\S_{g}^c} b(\mu)^{(n-s-2)/2} 
\Big[ \int_{\S_{g}^c}b(\lambda^\prime)^{n-s-1}
 \Big( \int_{g^* \cap \Omega_{e,f}} \hskip -5pt| u_{G(q, \lambda^\prime)} | \, dq \Big)^2 
\, d\lambda^\prime \Big]^{1/2} \, d \mu\\
 &\le c h^{j+s+1-n/2}  \Big[ \int_{\S_{g}^c}b(\lambda^\prime)^{n-s-1}
 \Big( \int_{g^* \cap \Omega_{e,f}} | u_{G(q, \lambda^\prime)} | \, dq \Big)^2 
\, d\lambda^\prime \Big]^{1/2}, 
\end{align*}
where we used that for $s < m \le n$, $(n-s-2)/2 \ge -1/2$, and hence
the integral with respect to $\mu$ is finite.  To complete the
argument, we apply the Cauchy-Schwarz inequality to the integral over
$g^* \cap \Omega_{e,f}$. Since the volume of $g^* \cap \Omega_{e,f}$
is $\textrm{O}(h^{n-s-1})$, we obtain
\begin{multline*}
\| L_{g}^*[b^{-j} R_{e.f}^ku] \|_{L^2(\Omega_f)} 
\le  c h^{j}  \Big[ h^{s+1}\int_{\S_{g}^c}b(\lambda)^{n-s-1}
  \int_{g^* \cap \Omega_{e,f}} | u_{G(q, \lambda} |^2 \, dq  
\, d\lambda \Big]^{1/2}
\\
\le c h^{j} \| u \|_{L^2(\Omega_{e,f})} \le c h^{j} \| u \|_{L^(\Omega_f^E)}.
 \end{multline*}
 This complete the verification of \eqref{keyterm-bound} when $g \neq
 \emptyset$.

 When $g = \emptyset$, then $L_g^*[b^{-j} R_{e,f}^ku] =0$ for $j <k$.
 When $j = k$, we have
\begin{equation*}
L_g^*[b^{-j} R_{e,f}^ku] = (R_{e,f}^ku)_0
= \int_{\Omega_{e,f}} (\Pi_j G^* u)_0 \wedge z_{e,f}
= \int_{\Omega_{e,f}} u_y \wedge z_{e,f}.
\end{equation*}
Hence, by the bound on $z_{e,f}$ given in Lemma~\ref{lem:zbound}, we have
\begin{equation*}
\| L_{\emptyset}^*[b^{-j} R_{e.f}^ku] \|_{L^2(\Omega_f)}
\le c h^{n/2} \Big|\int_{\Omega_{e,f}} u_y \wedge z_{e,f}\Big|
\le c h^{k} \| u \|_{L^2(\Omega_{e,f})},
\end{equation*}
which shows that \eqref{keyterm-bound} also holds in this case.
As a consequence, we have established the bound \eqref{Cmfkbound}.
\end{proof} 

\bibliographystyle{amsplain}


\end{document}